\documentclass[10pt]{amsart}

\usepackage{amsmath, amsthm, amssymb, graphicx, enumitem, everypage, datetime, getfiledate} 
\usepackage{wrapfig}

\newcounter{citedtheorems}

\numberwithin{equation}{section}

\newcounter{thestep}
\newcounter{thebigstep}

\newtheorem{defn}{Definition}[section]
\newtheorem{theorem}[defn]{Theorem}

\newtheorem*{theorem-m}{Theorem \ref{main-theorem}}
\newtheorem*{thm-p2a}{Theorem \ref{t:p2a}}

\newtheorem*{thm-seq}{Theorem \ref{t:seq}}

\newtheorem*{thm-m}{Main Theorem}
\newtheorem*{theorem-abs1}{Theorem \ref{ind-theorem}}
\newtheorem*{theorem-abs2}{Theorem \ref{a23}}
\newtheorem*{theorem-abs3}{Theorem \ref{ind-new}}
\newtheorem*{theorem-abs4}{Theorem \ref{m1}}
\newtheorem*{thm-x}{Theorem}
\newtheorem{thm-lit}[citedtheorems]{Theorem}
\newtheorem{defn-lit}[citedtheorems]{Definition}
\newtheorem{fact-lit}[citedtheorems]{Fact}
\newtheorem{fact}[defn]{Fact}
\newtheorem{cor}[defn]{Corollary}
\newtheorem{thesis}[defn]{Thesis}
\newtheorem{defn-claim}[defn]{Definition/Claim}

\newtheorem*{defn-in}{Definition \arabic{section}.\arabic{equation}}

\newtheorem*{claim-in}{Claim \arabic{section}.\arabic{equation}}

\newtheorem{concl}[defn]{Conclusion}
\newtheorem{conv}[defn]{Convention}
\newtheorem{claim}[defn]{Claim}
\newtheorem{subclaim}[defn]{Subclaim}
\newtheorem{lemma}[defn]{Lemma}
\newtheorem{obs}[defn]{Observation}
\newtheorem{rmk}[defn]{Remark}

\newtheorem{ntn}[defn]{Notation}
\newtheorem{disc}[defn]{Discussion}

\newtheorem{qst}[defn]{Question}

\newcommand{\lost}{\L os' }
\newcommand{\los}{\L os }
\newcommand{\br}{\vspace{2mm}}

\newcommand{\eff}{\mathcal{F}}
\newcommand{\gee}{\mathcal{G}}

\newcommand{\tlf}{\trianglelefteq}
\newcommand{\tlfn}{\triangleleft}

\newcommand{\rn}{\operatorname{range}}
\newcommand{\cf}{\operatorname{cof}}

\newcommand{\dom}{\operatorname{dom}}

\newcommand{\mcv}{\mathcal{V}}

\newcommand{\uu}{\mathcal{U}}

\newcommand{{\xw}}{\mathbf{w}}

\newcommand{\vv}{\mathbf{V}}

\newcommand{\crs}{\operatorname{cs}}
\newcommand{\bc}{\mathbf{c}}

\newcommand{\fdp}{f}

\newcommand{\ii}{\mathbf{i}}

\newcommand{\otp}{\operatorname{otp}}
\newcommand{\pr}{\operatorname{Pr}}
\newcommand{\prp}{{\operatorname{Pr}^\one}}
\newcommand{\pri}{{\operatorname{Pr}^{\ii}}}
\newcommand{\zero}{0}
\newcommand{\one}{1}
\newcommand{\ver}{\operatorname{vert}}
\newcommand{\supp}{\operatorname{supp}}
\newcommand{\tv}{\operatorname{tv}}
\newcommand{\ndx}{\operatorname{ind}}
\newcommand{\code}{\operatorname{Code}}

\newcommand{\mct}{\mathcal{T}}

\newcommand{\lgn}{\operatorname{lg}}

\newcommand{\de}{\mathcal{D}}

\newcommand{\ts}{\mathbf{S}}

\newcommand{\trv}{\mathbf{t}} 

\newcommand{\jj}{\mathbf{j}}
\newcommand{\mcp}{\mathcal{P}}

\newcommand{\ba}{\mathfrak{B}}

\newcommand{\lao}{[\lambda]^{<\aleph_0}}

\newcommand{\rstr}{\upharpoonright}
\newcommand{\mci}{\mathcal{I}}
\newcommand{\mcin}{{\mci_0}}
\newcommand{\mciy}{{\mci_1}}

\newcommand{\bad}{\mathcal{P}}

\newcommand{\vp}{\varphi}

\newcommand{\ma}{\mathbf{a}}
\newcommand{\mb}{\mathbf{b}}
\newcommand{\mc}{\mathbf{c}}

\newcommand{\mx}{\mathbf{x}}

\newcommand{\fin}{\operatorname{FI}}

\newcommand{\trg}{T_{\mathbf{rg}}}

\newcommand{\ap}{\operatorname{AP}}

\newcommand{\OP}{\operatorname{OP}}

\newcommand{\bx}{\mathbf{x}}
\newcommand{\by}{\mathbf{y}}

\title{Keisler's order has infinitely many classes}

\author{M. Malliaris and S. Shelah}\thanks{\emph{Thanks:}
Malliaris was partially supported by a Sloan research fellowship, by NSF grant DMS-1300634, 
and by a research membership at MSRI funded through NSF 0932078 000 (Spring 2014). 
Shelah was partially supported by European Research Council grant 338821. 
This is paper 1050 in Shelah's list.}

\address{Department of Mathematics, University of Chicago, 5734 S. University Avenue, Chicago, IL 60637, USA} 
\email{mem@math.uchicago.edu}
\address{Einstein Institute of Mathematics, Edmond J. Safra Campus, Givat Ram, The Hebrew
University of Jerusalem, Jerusalem, 91904, Israel, and Department of Mathematics,
Hill Center - Busch Campus, Rutgers, The State University of New Jersey, 110
Frelinghuysen Road, Piscataway, NJ 08854-8019 USA}
\email{shelah@math.huji.ac.il}
\urladdr{http://shelah.logic.at}

\date{\today}

\begin{document}



\begin{abstract} 
We prove, in ZFC, that there is an infinite strictly descending chain of classes of theories in Keisler's order. 
Thus Keisler's order is infinite and not a well order. Moreover, this chain occurs within the simple unstable theories, considered model-theoretically tame.   
Keisler's order is a central notion of the model theory of the 60s and 70s which compares first-order theories (and implicitly ultrafilters) 
according to saturation of ultrapowers.  
Prior to this paper, it was long thought to have finitely many classes, linearly ordered. The model-theoretic complexity we find is witnessed by a very natural class of theories, 
the $n$-free $k$-hypergraphs studied by Hrushovski.  This complexity reflects the difficulty of amalgamation and appears orthogonal to forking. 
\end{abstract}



\maketitle



A significant challenge to our understanding of unstable theories in general, and simple theories in particular, has been 
the apparent intractability of the problem of Keisler's order. Defined in 1967 \cite{keisler}, Keisler's order uses saturation of ultrapowers 
to compare the complexity of any two complete, countable first-order theories (basic model-theoretic objects of study).  
The recent articles \cite{MiSh:E74}, \cite{moore} explain the fundamental interest of the classification program 
arising from: 

\begin{qst}[Keisler 1967] Determine the structure of Keisler's order.
\end{qst}


The structure of Keisler's order on the stable theories was settled in Chapter VI of the second author's 1978 book \cite{Sh:c}. 
For a long time, there was little progress on the unstable case; \cite{MiSh:996} gives some history.  
This was due both to the difficulty of constructing ultrafilters and the state of our understanding of unstable theories. 
In the last few years, there has been very substantial progress in this area
(Malliaris \cite{mm-thesis}-\cite{mm5}, Malliaris and Shelah \cite{MiSh:996}-\cite{MiSh:1030}).

Notably, 
Keisler's order was thought to be finite, with perhaps four classes, whose identities were suggested in \cite{Sh:c} Problem VI.0.1.  
After \cite{MiSh:999} and \cite{MiSh:1030}, the number was at least five. 
In the present paper, we leverage the ZFC theorems of \cite{MiSh:1030} to prove, 
nearly fifty years after Keisler introduced the order, that: 

\begin{thm-m}[Theorem \ref{t:seq} below, in ZFC]
There is an infinite strictly descending sequence of simple unstable theories in Keisler's order. \emph{Thus}, Keisler's order 
is infinite and not a well order. 
\end{thm-m} 

The proof builds on a breakthrough of \cite{MiSh:1030}, the isolation of `explicit simplicity' and the corresponding construction of a family of so-called optimal ultrafilters.  
In some sense, such ultrafilters are to simple theories as Keisler's good ultrafilters \cite{keisler-1} are to all theories. 
The main constructions of \cite{MiSh:1030} allowed one of the cardinal parameters to be either $\aleph_0$ or an uncountable compact cardinal. 
However, the technology of that paper applies directly, in ZFC, to simple theories which have relatively little forking. 

This result revises both the picture of Keisler's order and the picture of simple theories. 
Simple theories, a generalization of stable theories which include the random graph and pseudofinite fields, 
are an active area of model-theoretic research as well as a fertile
interface for applications of model theory to geometry and combinatorics \cite{ChHr}, \cite{hr}, \cite{hrushovski}. 
Though simple theories retain some of the good behavior and structure theory available in stability, 
many basic questions about how simple theories may differ essentially from stable theories had remained open. 

Theorem \ref{t:seq} can be understood as showing that the asymptotic structure of simple theories is 
significantly more complex than that of stable theories (see diagram below).   
Moreover, the theories in question are low, with trivial forking.  
Building on \cite{MiSh:1030}, our methods here initiate a program of classifying the simple theories according to the complexity 
of coloring, that is, of the range $\mu$ of a coloring function $G$ as it is used in \S 5 below. 

The key model-theoretic objects in our proof are the generic $m$-free $k$-hypergraphs, \ref{d:m-k} below. The results 
from combinatorial set theory have to do with existence of free subsets in set mappings, combined with a major set-theoretic advance 
already mentioned,  
the existence theorem for so-called optimal ultrafilters recently proved by the authors in \cite{MiSh:1030} (see \S \ref{s:various} below). 

\begin{defn} \label{d:m-k}
$T_{m,k}$ is the model completion of the theory with one symmetric irreflexive $(k+1)$-ary relation $R$ and no complete graphs
on $m+1$ vertices. 
We say $(m,k)$ is \emph{nontrivial} to mean that $m > k \geq 2$, and
say that $T_{m,k}$ is nontrivial if $(m,k)$ is. We will assume $T_{m,k}$ is nontrivial unless otherwise stated.
\end{defn}

When $m>k=1$ (so the edge relation is binary), the theory is non-simple and in fact $SOP_3$, but in the case of hyperedges the situation is different. 

\begin{thm-lit} \emph{(Hrushovski \cite{h:letter})}
For $m > k \geq 2$, the theory $T_{m,k}$ is simple with trivial forking. 
\end{thm-lit}

Our argument will be guided by the following informal thesis. 

\begin{thesis} \label{m:thesis} The theories $T_{m+r, k+r}$ become in some sense less complicated $($closer to the random graph$)$  
as $r \rightarrow \infty$. 
\end{thesis}

Note that this thesis does not yet account for each coordinate growing separately.

\begin{disc} \emph{(of \ref{m:thesis})}
The ``forbidden configurations'' in the graphs $T_{m,k}$ are essentially barriers to amalgamation. In the Rado graph $\trg$, or equivalently in the
random hypergraphs $T_n$ for each $n<\omega$, there are no such barriers. 
As $r \rightarrow \infty$, the freedom to extend given configurations persists for up to $r$ steps before being impeded by an 
amalgamation constraint.
\end{disc}



%


Finally, what are the implications of our main theorem for Keisler's order? 

\begin{center}
\includegraphics[width=65mm]{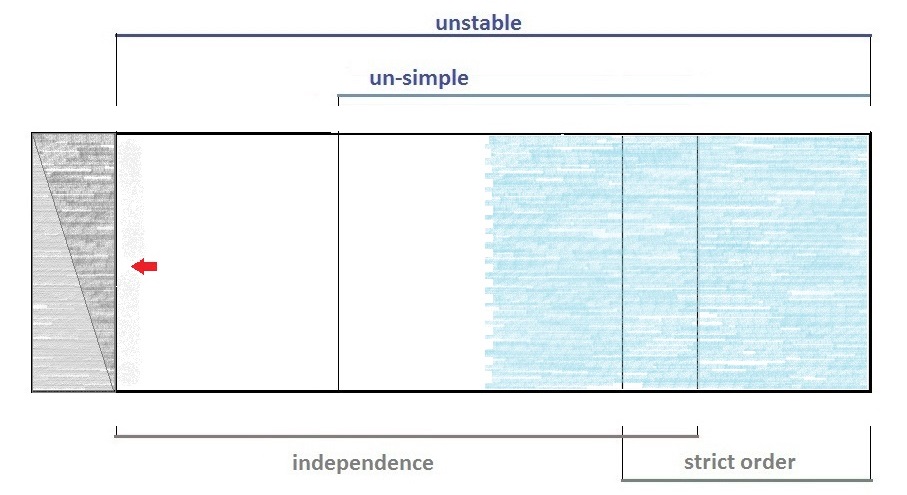} 
\end{center}

As the diagram indicates, Theorem \ref{t:seq} -- drawn here according to its full statement in Section \S \ref{s:infinite} below -- 
may be understood in parallel to the results of \cite{MiSh:998} proving that any theory with the strong tree property $SOP_2$
 is already maximal in Keisler's order.  
The unstable theories admit a natural structure/randomness dichotomy; any unstable theory either contains something akin to a bipartite random graph 
(the ``independence property'') or akin to a strict linear ordering (the ``strict order property''), and possibly both. From the point of view of Keisler's order, 
all linear orders, and indeed any theory with the much weaker property $SOP_2$, are in the maximum class, indicated by the large shaded region on the right. 
The infinite descending sequence of theories we find occurs in a comparatively tiny region of the class of independent or random theories, indicated by the arrow. 
This region had appeared, to previous instruments, to be relatively tame. This new picture supports the striking suggestion, hinted at already in \cite{mm5} and \cite{MiSh:998}, 
that Keisler's order is tied up with many such structure/randomness dichotomies within instability. Moreover, while linear order (``structure'') 
is maximally complex, 
this framework is able to illuminate significant but previously undetected jumps in the complexity of independence. 
The complexity we find here, using the lens of Keisler's order, has to do with failures of amalgamation and appears orthogonal to forking. 

We thank H. J. Keisler for comments on the manuscript.

\tableofcontents

\section{Preparation}

Our model-theoretic approach is guided by the framework of \cite{MiSh:1030} and in particular its program of 
stratifying the complexity of simple theories according to their so-called \emph{explicit simplicity}. 
As explained there, to capture the problem of realizing types in ultrapowers it is helpful to remember that 
at each index model, \lost theorem may guarantee that the `projections' of 
various finite fragments of the type are correct  
but it will not, in general, preserve their relative position. 
An informative translation of the complexity which may arise in such projections is to ask: 
given a type $p \in \ts(N)$, $||N|| = \lambda$ not forking over some small $M_*$, when can we color the finite pieces of $p$ (or more correctly, 
sufficiently closed sets containing them) with $\mu$ colors so that any time we move finitely many pieces of the same color by 
piecewise automorphisms which are the identity on $M_*$, 
agree on common intersections and introduce no new forking, the union of the images is a consistent partial type? 

The main theorems of the present paper will imply that for each finite $k \geq 2$, for $T_{k+1, k}$ it is necessary and sufficient to have $\mu$ colors 
when $\lambda = \mu^{+k}$. To see that, for instance, the tetrahedron-free three-hypergraph $T_{3,2}$ requires multiple colors,  
consider a type $\{ R(x,a,b), R(x,b,c), R(x,a,c) \}$ in the monster model, where ${\models \neg R(a,b,c)}$ but piecewise automorphisms of 
$\{ R(x,a,b) \}$, $\{ R(x,b,c) \}$, $\{ R(x,a,c) \}$ may move the parameters onto a triangle. 
 
\br 

Before making further remarks on strategy, we review a family of classical results on set mappings. 
Proofs of Theorems \ref{t:yes}, \ref{t:no}, \ref{t:maybe} may be found in Erd\"os, Hajnal, Mate, and Rado \cite{ehmr}, as noted.
We use $\lambda, \kappa, \mu$ for infinite cardinals and $k, \ell, m, n$ for integers. 

\begin{defn}
Let $m, n$ be integers, $\alpha$ an ordinal, and $\lambda, \mu$ infinite cardinals. 
\begin{enumerate}
\item We say $F: [\lambda]^m \rightarrow [\lambda]^{<\mu}$ is a \emph{set mapping} if $F(x) \cap x = \emptyset$ for $x \in [\lambda]^m$. 
\item We say the set $X \subseteq \lambda$ is \emph{free} with respect to $F$ if $F(x) \cap X = \emptyset$ for every $x \in [X]^m$.
\end{enumerate}
\end{defn}

\begin{ntn} We write:
\[ (\lambda, m, \mu) \longrightarrow n \]
to mean that for every set mapping $F : [\lambda]^m \rightarrow [\lambda]^{<\mu}$ there is a set $X$ of size $n$ which is free with respect to $F$, 
and write 
\[ (\lambda, m, \mu) \not\longrightarrow n \]
to mean that for some set mapping $F : [\lambda]^m \rightarrow [\lambda]^{<\mu}$ no set of size $n$ is free with respect to $F$. 
\end{ntn}

A celebrated theorem of Sierpi\'nski \cite{sier} states that the continuum hypothesis holds precisely when $\mathbb{R}^3$ admits a decomposition 
into three sets $A_x$, $A_y$, $A_z$ such that for $w=x,y,z$, $A_w$ intersects all lines in the direction of the $w$-axis in finitely many points. 
In other words, this property characterizes $\aleph_1$. 
Kuratowski and Sierpi\'nski then characterized all $\aleph_n$s via set mappings: 

\begin{thm-lit} \emph{(see \cite{ehmr} Theorem 46.1)} \label{t:yes}
For any $m < \omega$ and ordinal $\alpha$ we have that
\[  \left( \aleph_{\alpha + m }, m, \aleph_\alpha \right) \longrightarrow m+1 \]
\end{thm-lit}

\begin{thm-lit} \emph{(see \cite{ehmr} Theorem 45.7)} \label{t:no}
For any $m < \omega$ and ordinal $\alpha$ we have that
\[  \left( \aleph_{\alpha + m }, m + 1, \aleph_\alpha \right) \not\longrightarrow m+2 \]
\end{thm-lit}

\begin{cor} \label{d:mnt} \emph{(Monotonicity)}
Given $m_0 \leq m \leq m_1$, $n_0 \leq n \leq n_1$, 
\begin{enumerate}
\item[(a)] if $n > m$ and $(\lambda, m, \mu) \longrightarrow n$  then  $(\lambda, m_0, \mu) \longrightarrow n_0$. 
\item[(b)] if $n_1 > m_1$ and $(\lambda, m, \mu) \not\longrightarrow n$ then $(\lambda, m_1, \mu) \not\longrightarrow n_1$. 
\end{enumerate}
\end{cor}

\begin{proof}
(a) Suppose we are given a set mapping $F: [\lambda]^{m_0} \rightarrow [\lambda]^{<\mu}$. Define $F^\prime : [\lambda]^m \rightarrow [\lambda]^{<\mu}$ by 
$F^\prime(u) = \bigcup \{ F(u^\prime) : u^\prime \in [u]^{m_0} \} \setminus u$.  Clearly $F^\prime$ is a set mapping. 
Since $(\lambda, m, \mu) \rightarrow n$, there is a free set $v_*$ for $F^\prime$, $|v_*| = n$. Fix $v_{**} \subseteq u$ with $|v_{**}| = n_0 \leq n$. 
Let us check that $v_{**}$ is as required for $F$. Let $u \in [v_{**}]^{m_0}$ and $\alpha \in v_{**}\setminus u$. Since $F$ is a set mapping, it suffices to prove that 
$\alpha \notin F(u)$. Towards this, choose $w \subseteq v_* \setminus (u \cup \{ \alpha \})$ of size $m - m_0$, which is possible because 
$| v_* \setminus (u \cup \{ \alpha \} ) | = n - (m_0 + 1) = (n -1) - m_0 \geq m - m_0$. So $u \cup w \subseteq v_*$ and $|u \cup w| = m_0 + (m-m_0) = m$ 
elements and $\alpha \in v_* \setminus (u \cup w)$. As $v_*$ was chosen to be a free set for $F^\prime$, necessarily $\alpha \notin F^\prime(u \cup w)$.
Now recalling the definition of $F^\prime$, $F^\prime(u \cup w) = \bigcup \{ F(x) : x \in [u \cup w]^{m_0} \} \setminus ( u \cup w )$. Since 
$\alpha \notin (u \cup w)$ and we know that $\alpha \notin F^\prime(u \cup w)$, we conclude $\alpha \notin F(u)$ as desired. 

(b) This holds by the contrapositive of (a), i.e. (a) applied to $m_1, n_1, m, n$ instead of $m, n, m_0, n_0$. 
\end{proof}

The general situation for free sets of large finite size relative to $m$ 
is less clear. For instance, it is known that: 

\begin{thm-lit} \emph{(see \cite{ehmr} Theorem 46.2)} \label{t:maybe}
For any $n < \omega$ and ordinal $\alpha$ we have that
\begin{itemize}
\item \emph{(Hajnal-Mate)} $(\aleph_{\alpha + 2}, 2, \aleph_\alpha) \longrightarrow n$

\item \emph{(Hajnal)}  $(\aleph_{\alpha + 3}, 3, \aleph_\alpha) \longrightarrow n$

\end{itemize}
\end{thm-lit}

However, we note there are also consistency results.\footnote{We carry out the present proof entirely in ZFC. It will be very interesting to see whether future work will show 
such independence results to also be reflected in the model-theoretic structure of simple theories, 
or whether the connection goes no further than what we develop here.} In the following theorem, 
$\tau(n+1)$ is the least natural number such that $\tau(n+1) \rightarrow (\tau(n), 7)^5$. (Further results are in a forthcoming paper \cite{Sh:F1406}.)

\begin{thm-lit} \emph{(Komj\'ath and Shelah 2000 \cite{KoSh:645}, Theorem 1)} \label{mst}
There is a function $\tau: \omega \rightarrow \omega$ such that whenever $\mu$ is regular, $n<\omega$, $\lambda = \mu^{+n}$, $\mu = \mu^{<\mu}$, and 
$\bigwedge_{\ell < n} 2^{\mu^{+\ell}} = \mu^{+\ell+1}$, 
for some $(<\mu)$-complete $\mu^{+(n+1)}$-c.c. forcing notion $\mcp$ of cardinality $\lambda$ collapsing no cardinals, in $\vv^{\mcp}$ 
we have $2^\mu = \mu^{+n}$ and  
there is a set mapping $F : [\lambda]^4 \rightarrow [\lambda]^{<\mu}$ with no free subset of size $\tau(n)$. 

In symbols, under these assumptions,
\[ \left( \mu^{+n}, 4, \mu \right) \not \longrightarrow \tau(n) \]
\end{thm-lit}

\br

We now briefly motivate our use of these theorems for realizing and omitting types. A first adjustment is that we would 
like to enclose fragments of types in suitable larger parameter sets, so we will want to replace the condition 
$x \cap F(x) = \emptyset$ in the definition of set mapping with the condition that $x \subseteq F(x)$ as in \ref{d:1.4}(1) and also to 
replace ``not free'' with \ref{d:1.4}(2).

\begin{defn} \label{d:1.4}
Let $k < n$ be integers, $\alpha$ an ordinal, and $\lambda, \theta$ infinite cardinals. 
\begin{enumerate}
\item We say $F: [\lambda]^k \rightarrow [\lambda]^{<\theta}$ is a \emph{strong set mapping} if $x \subseteq F(x)$ for $x \in [\lambda]^k$. 
\item We say the set $X \in [\lambda]^n$ is \emph{covered} with respect to $F$ if there exists $x \in [X]^k$ such that $X \subseteq F(x)$. 
\end{enumerate}
\end{defn}

Briefly, working in one of the theories $T_{k+1, k}$ ($k \geq 2$), we will want to 
associate to each finite subtype a larger `enveloping' set, and to color these envelopes 
in such a way that within any fixed color class, any time 
a near-forbidden configuration (e.g. an $R$-triangle on the parameters) appears it must 
already be contained in one of the associated envelopes. In the course of our analysis, 
we will be able to ensure the individual envelopes correspond to consistent partial types over submodels. 
As will be explained, this property of absorbing forbidden configurations will then give sufficient 
leverage for a proof (by contradiction) of amalgamation within each color class. 
The right formalization for our present arguments is the following. We will use the case $\theta = \aleph_0$.  

\begin{defn} \label{x2} 
Let $\pr_{n,k}(\lambda, \mu, \theta) = \pr^0_{n,k}(\lambda, \mu, \theta)$ be the assertion that 
we can find $G: [\lambda]^{<\theta} \rightarrow \mu$ such that: 
\begin{quotation}
if $w \in [\lambda]^n$, 
$\overline{u} = \langle u_v : v \in [w]^{k} \rangle$, 
$v \in [w]^{k}$ implies $v \subseteq u_v \in [\lambda]^{<\theta}$ 
and $G \rstr \{ u_v : v \in [w]^{k} \}$ is constant, 
then for some $v \in [w]^{k}$ we have $w \subseteq u_v$. 
\end{quotation}
\end{defn}

The full proof of saturation involves a detailed analysis of model-theoretic amalgamation problems arising in ultrapowers with 
Definition \ref{x2} as a key ingredient. 

A crucial point of this definition appears in the existence proof, Lemma \ref{z4} below: for the hypergraphs 
in question we may always take $\mu < \lambda$, and in fact, the subscript $k$ is tied to the 
cardinal distance of $\lambda$ and $\mu$.  
The $(\lambda,\mu)$-\emph{optimal} (or: perfect) ultrafilters of \cite{MiSh:1030} play an important role, as will be
explained in due course. 

If failures of freeness, which is to say of covering, help with saturation, when will existence of free sets yield omitted types? 
A priori, given a model $N = (\lambda, R) \models T_{n,k}$, 
we cannot directly apply Theorem \ref{t:yes} to omit a type, 
because that theorem does not guarantee that the free set will occur on an $R$-complete graph.  The right 
analogue for non-saturation will be the following. 

\begin{lemma}[proved in \S \ref{s:properties}] \label{f2ax} 
Suppose that $n > k \geq 2$ and $(\lambda, k, \mu) \rightarrow n$. 
Then there is a model $M$ of $T_{n,k}$ of size $\geq \lambda$, 
and $\lambda$ elements of its domain 
$\langle b_\alpha : \alpha < \lambda \rangle$, such that 
writing 
\[ \bad = \{ w \in [\lambda]^n : (\forall v \in [w]^{k})(M \models R(\overline{b}_v) ) \} \] 
we have that for any strong set mapping  $ F: [\lambda]^{k} \rightarrow [\lambda]^{<\mu} $,  
for some $w \in \bad$ 
\[ (\forall v \in [w]^{k})(w \not\subseteq F(v)).  \]
\end{lemma}


\noindent 

For orientation, the reader may now wish to read the statement of Theorem \ref{t:p2a}, as well as of the Main Theorem 
\ref{t:seq}. Keisler's order is defined in \S \ref{s:infinite}. Further background on Keisler's order and 
saturation of ultrapowers appears in \cite{MiSh:E74} and in the introduction to \cite{MiSh:1030}.  
Earlier sources are \cite{keisler}, \cite{kochen}. 

\br

We now turn to the proofs. 

\section{Key covering properties} \label{s:properties}

In this section, we give the the existence proof 
corresponding to Definition \ref{x2} above, Lemma \ref{z4}. We also prove 
Lemma \ref{f2ax} from p. \pageref{f2ax} above. ``$\pr$'' abbreviates ``property.'' 
Locally in this section, we will refer to the property from \ref{x2} as 
$\pr^0_{n,k}(\lambda, \mu, \theta)$ to distinguish it from 
the weaker variant $\prp_{n,k}(\lambda, \mu, \theta)$ defined below.  
We will establish results for both $\pr^0$ and $\pr^1$ in this section, 
although only $\pr^0$ is central for our proofs. 
(In all later sections in the paper, $\pr$ means $\pr^0$, 
as will be stated in Convention \ref{c-superscript}.)

\begin{defn} \label{x2a} 
Let $\prp_{n,k}(\lambda, \mu, \theta)$ be the statement that:

if $N = (\lambda, R) \models T_{n,k}$ then we can find $G: [\lambda]^{<\theta} \rightarrow \mu$ such that: 
\begin{quotation}
if $w \in [\lambda]^n$, $(w, R\rstr w)$ is a complete hypergraph, 
\\ $\overline{u} = \langle u_v : v \in [w]^{k} \rangle$, 
$v \in [w]^{k} \implies v \subseteq u_v \in [\lambda]^{<\theta}$ 
\\ and $G \rstr \{ u_v : v \in [w]^{k} \}$ is constant, 
\\ then for some $v \in [w]^{k}$ we have $w \subseteq u_v$. 
\end{quotation}
\end{defn}

\begin{obs} \label{z3a} \emph{ }
\begin{enumerate}
\item If $\theta_1 \geq \theta_2$, $\lambda_1 \geq \lambda_2 \geq \mu_2 \geq \mu_1$ and $\ii \in \{ \zero, \one \}$ then 
\[ \pri_{n,k}(\lambda_1, \mu_1, \theta_1) \implies \pri_{n,k}(\lambda_2, \mu_2, \theta_2). \] 
\item $\pr^0_{n,k}(\lambda, \mu, \theta) \implies \pr^1_{n,k}(\lambda, \mu, \theta)$.  
\end{enumerate}
\end{obs} 

\begin{proof}
(1)  If $\ii = 1$, let $M_1 = (\lambda_1, R)$ and let $M_2 = M_1 \rstr \lambda_2 = (\lambda_2, R)$. 
If $G_1 : [\lambda_1]^{<\theta_1} \rightarrow \mu_1$ is suitable for $\theta_1$ (and the model $M_1$ if $\ii = 1$), 
then $G_2 = G_1 \rstr [\lambda_2]^{<\theta_2}$ is suitable for $\theta_2$ (and $M_2$ if $\ii = 1$) and has range $\subseteq \mu_1 \subseteq \mu_2$. 
(2) is immediate. 
\end{proof}

Although we will not use it, note that $\pr^0_{n,k}(\lambda, \mu, \theta)$ is preserved by forcing not collapsing $\lambda$, $\mu$, or $\theta$. 
The next Claim \ref{z3} reduces to the case where $\theta = \aleph_0$, and in fact, isolates a finite bound.

\begin{claim} \label{z3}
If $k \geq 2$, $n = k+1$, $\mu = \mu^{<\theta} = \cf(\mu)$, $\lambda = \mu^{+k}$, and $\ii \in \{ \zero, \one \}$,
then 
\[ \pri_{n,k} (\lambda, \mu, \theta) ~\mbox{  iff   }~\pri_{n,k}(\lambda, \mu, \aleph_0) ~\mbox{   iff   }~ \pri_{n,k}(\lambda, \mu,~ k{\binom{n}{k}}^2+1) \]
\end{claim}

\begin{proof}  There are three clauses. The first implies the second implies the third by Observation \ref{z3a}(1).
We will prove the main case we use, second implies first, but the proof will show that in fact the third implies the second (and thus the first). 

First, fix $M$: if $\ii = \one$, let $M$ be the given model on $\lambda$, and if $\ii = \zero$, let $M$ be the complete hypergraph on $\lambda$.
Note that $\pr^{\zero}$ implies $\pr^{\one}$ by \ref{z3a}(2).  
We will show that $\pri_{n,k}(\lambda,\mu,\aleph_0)$ for $M$ implies $\pri_{n,k}(\lambda,\mu,\theta)$ for $M$. 

We start with $G_1 : [\lambda]^{<\aleph_0} \rightarrow \mu$ witnessing $\pri_{n,k}(\lambda, \mu, \aleph_0)$, and use it to build a function $G_2 : [\lambda]^{<\theta} \rightarrow \mu$ 
as follows. The notation ``$\OP_{y,x}$'' is well defined when $x,y \subseteq \operatorname{Ord}$ have the same order type and denotes the one-to-one and onto order preserving map from $x$ to $y$.  Define a two-place relation $E_*$ on $[\lambda]^{<\theta}$ by: 
\begin{align*}
u_1 E_* u_2 ~~ \iff ~~ & \otp(u_1) = \otp(u_2)  \mbox{ and } \\
                       & \OP_{u_2, u_1}: u_1 \rightarrow u_2 \mbox{ commutes with $G_1$, i.e. }\\
                       & \mbox{if $v \in [u_1]^{<\aleph_0}$ then $G_1(v) = G_1(\{ \OP_{u_2, u_1}(\alpha) : \alpha \in v \} )$} 
\end{align*}
Clearly $E_*$ is an equivalence relation with $\leq \mu^{<\theta} = \mu$ equivalence classes. 
Let $G_2: [\lambda]^{<\theta} \rightarrow \mu$ capture the $E_*$-equivalence classes, i.e. $G_2(u_1) = G_2(u_2) \iff u_1 E_* u_2$. 
It remains to verify that: 
\begin{quotation}
if $w \in [\lambda]^n$, $(w, R\rstr w)$ is a complete hypergraph, 
\\ $\overline{u} = \langle u_v : v \in [w]^{k} \rangle$, 
$v \in [w]^{k} \implies v \subseteq u_v \in [\lambda]^{<\theta}$ 
\\ and $G_2 \rstr \{ u_v : v \in [w]^{k} \}$ is constant, 
\\ then for some $v \in [w]^{k}$ we have $w \subseteq u_v$. 
\end{quotation}
Suppose then that we are given such a $w \in [\lambda]^{n}$ and sequence $\langle u_v : v \in [w]^{k} \rangle$. 
Since $G_2$ is constant on $\langle u_v : v \in [w]^k \rangle$, the elements of this sequence are all $E_*$-equivalent 
and so have the same order type $\zeta$. For each $v \in [w]^{k}$ define $f_v : \zeta \rightarrow u_v$ to be the 1-to-1 and onto order preserving map. 
Now define $h: [w]^{k} \rightarrow  [\zeta]^{k}$ by:
\[ h(v) = \{   \epsilon < \zeta : f_v(\epsilon) \in v \} \in [\zeta]^{k} \]
so informally, $h(v)$ is the canonical preimage of $v$ in $\zeta$ (recall that $v \subseteq u_v$ by hypothesis so this is well defined). 
Let the union of these canonical preimages be  
\[ u_* = \bigcup \{ h(v) : v \in [w]^{k} \} \in [\zeta]^{<\aleph_0} \mbox{, in fact $\in ~[\zeta]^{\leq k\cdot\binom{n}{k}}$ } \]
At this point, the problem is finitary: for each $v \in [w]^{k}$, let $u^\prime_v = \{ f_v(\epsilon) : \epsilon \in u_* \}$.  
The sequence $\langle u^\prime_v : v \in [w]^k \rangle$ is a sequence of elements of $[\lambda]^{<\aleph_0}$, and 
by construction $v \subseteq u^\prime_v$ for each $v \in [w]^k$. Moreover, for each $v \in [w]^k$, 
$u^\prime_v \subseteq u_v$, and the original sets $\{ u_v : v \in [w]^{k} \}$ are pairwise $E_*$-equivalent. The definition of $E_*$ 
entails that pushing forward a finite set via an order preserving map doesn't change the value of $G_1$, and evidently each 
$u^\prime_v$ has the same preimage $u_*$. Thus 
$G_1 \rstr \{ u^\prime_v : v \in [w]^{k} \}$ is constant. 
As $G_1$ is a witness to $\pri_{n,k}(\lambda, \mu, \aleph_0)$, there is $v \in [w]^{k}$ for which $w \subseteq u^\prime_v$, {thus} $w \subseteq u_v$.  
This shows that $G_2$ witnesses $\pri_{n,k} (\lambda, \mu, \theta)$, which completes the proof. 
\end{proof}

We arrive to a fundamental lemma.  Appropriately, its proof works like a proof of amalgamating 
diagrams of models. 

\begin{lemma} \label{z4}
If $k \geq 2$, $n = k+1$, $\mu = \mu^{<\theta} = \cf(\mu)$, 
$\lambda \in [\mu,  \mu^{+{k-1}}]$, 
and $\ii \in \{ \zero, \one \}$, 
then $\pri_{n,k} (\lambda, \mu, \theta)$. 
\end{lemma}


\begin{proof}
By Claim \ref{z3}, without loss of generality, $\theta = \aleph_0$. 
Let $M$ be given, either as the model of $T_{n,k}$ on $\lambda$ given with the data of $\prp$ or else as the complete graph on $\lambda$, or recall 
Observation \ref{z3a}(2). 
We proceed in stages.

First, when $\uu \subseteq \lambda$, $G : [\uu]^{<\aleph_0} \rightarrow \mu$, say that the pair $(\uu, G)$ satisfies the requirements of $\pri$ to mean:

\begin{quotation}
if $w \in [\uu]^n$, $(w, R\rstr w)$ is a complete hypergraph, 
\\ $\overline{u} = \langle u_v : v \in [w]^{k} \rangle$, 
$v \in [w]^{k} \implies v \subseteq u_v \in [\lambda]^{<\aleph_0}$ 
\\ and $G \rstr \{ u_v : v \in [w]^{k} \}$ is constant, 
\\ then for some $v \in [w]^{k}$ we have $w \subseteq u_v$. 
\end{quotation}

Second, we define a family of approximations. 
For each $\ell = 0, \dots, k$ let $\ap_\ell$ be the set of $\ell$-approximations $\bx$, where 
\[ \bx = \langle \uu, \overline{\alpha}, \overline{G} \rangle \]
and these objects satisfy:
\begin{enumerate}[label=\emph{(\alph*)}]
\item $\uu \subseteq \lambda$, $|\uu| < \mu^{+\ell}$
\item $\overline{\alpha} = \langle \alpha_m : m < k - \ell \rangle$ is a sequence of ordinals 
from $\lambda \setminus \uu$ with no repetition
\item let $A = \{ \alpha_m : m  < k - \ell \}$
\item for $j < \lgn(\overline{\alpha})$, let $\uu_{\alpha_j} = \uu \cup \{ \alpha_m : m \neq j \}$
\item $\overline{G} = \langle G_\alpha : \alpha \in A \rangle$
\item for $\alpha \in A$, $G_\alpha : [\uu_\alpha]^{<\aleph_0} \rightarrow \mu$
\item for each $\alpha \in A$, $(\uu_{\alpha}, G_\alpha)$ satisfies the requirements of $\pri$ 
\\ in the sense of Step 0.
\item for each $\alpha, \beta \in A$, $G_\alpha$ and $G_\beta$ are compatible functions, meaning ,that 
\[ G_\alpha \rstr {\uu_{\alpha} \cap \uu_{\beta}} = G_\beta \rstr {\uu_{\alpha} \cap \uu_{\beta}} \]
\end{enumerate}

\br
\noindent Given $\bx \in \ap = \bigcup_\ell \ap_\ell$, we say that ``$G$ is a solution for $\bx$'' to mean that:
\begin{itemize}
\item $G$ is a function with domain $[\uu_\bx \cup \{ \alpha_{\bx, m} : m < k - \ell \}]^{<\aleph_0}$ which is into $\mu$
\item $G \supseteq G_{\bx, \alpha_m}$ for $m < k-\ell$
\item $(\uu_\bx \cup \{ \alpha_{\bx, m} : m < k - \ell \}, G)$ satisfies the requirements of $\pri$.
\end{itemize}

Third, we prove by induction on $\ell \leq k$ that each $\bx \in \ap_\ell$ has a solution. 

Case 1: $\ell = 0$. 
In this case, writing $\uu^+ = \uu_\bx \cup \{ \alpha_{\bx, m} : m < k \}$, we will need to extend the coloring to 
all elements of $[\uu^+]^{<\aleph_0}$.
On one hand, let $\langle u_i : i < i_* < \mu \rangle$ list $[\uu^+]^{<\aleph_0} \setminus \bigcup \{ [\uu_{\bx, m}]^{<\aleph_0} : m < k \}$ with no repetition: 
these are the $<\mu$ sets whose coloring needs to be determined. On the other hand, let $X = \bigcup \{ \rn(G_{\bx, \alpha_m}) : m < k \}$. 
As $\ell = 0$, $|X| < \mu$. Enumerate $\mu \setminus X$ as $\langle \gamma_\epsilon : \epsilon < \mu \rangle$, without repetition. 
Then there is enough room to trivially satisfy $\pri$ by avoiding further monochromaticity: simply 
define $G(u)$ to be $G_{\bx, \alpha_m}(u)$ if $\alpha_{\bx, m} \notin u \land m < k$ and otherwise to be $\gamma_\epsilon$ if $u = u_\epsilon$.
\br


Case 2: $\ell > 0$. In this case we build a solution by induction. 
Let $\langle \beta_i : i < |\uu_\bx| \rangle$ list $\uu_\bx$ with no repetition. 
For each $i < |\uu_\bx|$, let $\sigma_i = \{ \beta_j : j < i \} \cup \{ \alpha_{\bx, m} : m < k - \ell \}$ be the initial segment along 
with the new points. 

By induction on $i \leq |\uu_\bx| < \mu^{+\ell}$, we choose $G_i$ such that:
\begin{itemize}
\item $G_i$ is a function from $[\sigma_i]^{<\aleph_0}$ into $\mu$ 
\item $G_i$ satisfies the requirements in $\pr^{\zero}_{n,k}(\lambda, \mu)$
\item $G_i$ extends $G_j$ for $j < i$
\item $G_i$ is compatible with $G_{\bx, \alpha_m}$ for $m < k -\ell$
\end{itemize}

Subcase 1: $|\sigma_i| < n$, i.e. the initial segment is finite: trivial.

Subcase 2: $i$ is a limit. Since each function extends its predecessors, let $G_i = \bigcup_{j < i} G_j$. 

Subcase 3: $i = i_* + 1$ is an infinite successor.  
In this case we reduce to a previous amalgamation problem as follows.\footnote{The induction allows us to 
reduce the $n$-amalgamation problem over a set of size $\kappa$ to an $n+1$-amalgamation problem 
over a set of size less than $\kappa$.}   
Define $\by_i \in \ap_{\ell - 1}$ by:
\begin{quotation}
$\uu_{\by_i} = \{ \beta_j : j < i_* \}$

\noindent $\alpha_{\by_i, m}$ is $\alpha_{\bx, m}$ for $m < k-\ell$ and $\beta_{i_*}$ for $m = k-\ell < k-(\ell-1)$. 

\noindent $G_{\alpha}$ is given by $G_{\bx, \alpha_{\bx, m}}$ for $m < k-\ell$ and by $G_{i_*}$ (i.e., by the inner inductive hypothesis on $i$) 
for $m = k - \ell$.
\end{quotation}

This data defines an element $(\uu, \overline{\alpha}, \overline{G})$ of $\ap_{\ell - 1}$, so now 
use the (outer, i.e. on $\ell$) inductive hypothesis to complete the proof.
\end{proof}

\begin{claim} \label{c10}
Suppose $(\lambda, k, \mu) \rightarrow n$ and $k<n$. Then for any $F: [\lambda]^k \rightarrow [\lambda]^{<\mu}$, e.g. 
$F$ a strong set mapping, 
there is $w \in [\lambda]^n$ such that for all $u \in [w]^k$, $w \not\subseteq F(u)$. 
\end{claim}


\begin{proof}
If $F$ is a set mapping, this is immediate. 
If not, identify $\lambda$ with $\lambda \setminus \{ 0 \}$ and 
define $G: [\lambda]^k \rightarrow [\lambda]^{<\mu}$  by: 
$G(u) = (F(u) \setminus u)~ \cup ~\{ 0 \}$ [so that $G(u) \neq \emptyset$].  
Then $G$ is a set mapping, so there is some $w \in [\lambda]^n$ such that for all $u \in [w]^k$, 
$w \cap G(u) = \emptyset$. Thus for each $u \in [w]^k$, 
if $w \cap F(u) \neq \emptyset$ then $w \cap F(u) \subseteq u$.  
Since $|u| = k$, this means $|w \cap F(u) | \leq k < n = |w|$ so 
$w \not\subseteq F(u)$ as desired. 
\end{proof}

We now prove Lemma \ref{f2ax}, promised on p. \pageref{f2ax} above. 

\begin{proof}[Proof of Lemma \ref{f2ax}]
Let us define a model $N = (\lambda, R^N)$ by:
$R^N = \{ (\alpha_0, \dots, \alpha_k) : \alpha_i < \lambda \mbox{ and } i_1 < i_2 \leq k \implies \alpha_{i_1} \neq \alpha_{i_2} \mod n \mbox{ and } 
i_1 < i_2 \leq k \implies (\exists \beta)(\alpha_{i_1} < n\beta \leq \alpha_{i_2} \lor \alpha_{i_2} < n\beta \leq \alpha_{i_1} ) \}$. 
By definition $R^N$ is irreflexive, symmetric, and $(k+1)$-ary. Let us first show that if $w \in [\lambda]^{n+1}$ then $w$ is not a complete $R^N$-hypergraph. 
If $|w| = n+1$, for some $\alpha_1 \neq \alpha_2 \in w$ we have $\alpha_1 = \alpha_2 \mod n$. Choose $v \subseteq w$ such that $\alpha_1 \in v$, $\alpha_2 \in v$, 
and $|v| = k+1$. Then by definition $R$ cannot hold on $\{ \alpha : \alpha \in v \}$, so $w$ cannot be a complete $R$-hypergraph. So $N$ is a submodel of 
some $N^\prime \models T_{n,k}$ of cardinality $\lambda$.

Second, let us show that if $F: [\lambda]^k \rightarrow [\lambda]^{<\mu}$ is a given strong set mapping 
then for some $w \in \bad$ we have $(\forall v \in [w]^k)(w \not\subseteq F(v))$. 
First let $F_1: [\lambda]^k \rightarrow [\lambda]^{<\mu}$ be defined by: if $v = \{ \alpha_0, \dots, \alpha_{k-1} \} \in [\lambda]^k$ then 
\begin{align*}
 F_1(v) = &\{ \beta : \mbox{ for some } i_0, \dots, i_k < n \mbox{ and } \gamma < \lambda \\
 & \mbox{ we have } (\beta < n\gamma + n) \land (\gamma < n\beta + n)  \\
& \mbox{ and } \gamma \in F(\{ n\alpha_0 + i_0, \dots, n\alpha_{k-1}+i_{k-1} \}) \} 
\end{align*} 
As we have assumed $(\lambda, k, \mu) \rightarrow n$, by Claim \ref{c10} there is 
$w_1 = \{ \alpha^*_i : i < n \}$ such that $\alpha^*_0 < \cdots < \alpha^*_{n-1} < \lambda$ 
and for all $v \in [n]^k$,  $w_1 \not\subseteq F_1( \{ \alpha^*_\ell : \ell \in v \})$. 

Define $\beta_i = n \alpha^*_i + i$ for $i<n$ and let $w_2 = \{ \beta_i : i < n \}$, and let us show $w_2$ is as required for $F$. 
First, trivially $|w_2|=n$, as $|w_1| = n$. Second, by definition of $R^N$ above, $w_2$ is a complete $R^N$-hypergraph. 
Third, let us show that if $u \in [w_2]^k$ and $w_2 \subseteq F(u)$ then we get a contradiction. 
Let $v \in [n]^k$ be such  that  $u = \{ \beta_i : i \in v \}$. 
Then $u_1 := \{ \alpha^*_i : i \in v \} \in [w_1]^k$ 
and $w_1 \subseteq F(u_1)$. 
[This is because for each $\alpha^*_i$, $i<n$ we presently have $\beta_i \in F(u)$ so there is an element  
\[ \gamma \in F(\{ n\alpha^*_0 + 0, \dots, n\alpha^*_{k-1}+(k-1) \}) = F(\beta_0, \dots, \beta_{k-1}) \] 
such that $(\alpha^*_i < n\gamma + n) \land (\gamma < n\alpha^*_i + n)$, namely $\gamma = \beta_i$, 
which belongs to $F(\beta_0, \dots, \beta_{k-1})$
since $F$ is a strong set mapping.] This contradicts the choice of $w_1$. We have shown that for all 
$u \in [w_2]^k$, $w_2 \not\subseteq F(u)$, so $w_2$ is as required which completes the proof. 

\noindent \emph{Lemma \ref{f2ax}.} \hfill
\end{proof}

\begin{conv} \label{c-superscript}
In the remainder of the paper, 
\begin{enumerate}
\item ``$\pr$'' used without a superscript means $\pr^0$. 
\item All theories are complete and countable unless otherwise stated. 
\end{enumerate}
\end{conv}

\br
\section{Separation of variables and optimal ultrafilters} \label{s:various}

In this section we explain the advances from \cite{MiSh:999} and \cite{MiSh:1030}, Theorems \ref{t:separation} 
and \ref{t:exists} below, which frame the rest of the proof. 

The gold standard for saturation is the following class of ultrafilters, called good, introduced by Keisler \cite{keisler-1}. 
Keisler proved that good ultrafilters exist, assuming GCH, and Kunen eliminated the assumption of GCH \cite{kunen}.

\begin{defn} A filter $\de$ on $\lambda$ is called \emph{good} if every monotonic $f: [\lambda]^{<\aleph_0} \rightarrow \de$ 
has a multiplicative refinement. In other words, if $f$ satisfies $u \subseteq v$ implies $f(v) \subseteq f(u)$ then there is 
$g: [\lambda]^{<\aleph_0} \rightarrow \de$ such that $g(u) \subseteq f(u)$ for all finite $u$ and $g(u \cup v) = g(u) \cap g(v)$ for all finite $u,v$. 
\end{defn}

\begin{fact}[Keisler \cite{keisler}]
If $\de$ is a regular ultrafilter on $\lambda$, then $\de$ is good if and only if for every complete countable theory $T$ and
any $M \models T$, $M^\lambda/\de$ is $\lambda^+$-saturated. 
\end{fact}

We will informally say ``$\de$ is good for $T$'' to mean that for all\footnote{When $\de$ is regular, 
if $M \equiv N$ in a countable signature then $M^\lambda/\de$ is $\lambda^+$-saturated if and only if $N^\lambda/\de$ is $\lambda^+$-saturated, 
Keisler \cite{keisler} Corollary 2.1a.} $M \models T$,  $M^\lambda/\de$ is $\lambda^+$-saturated. 
Thus, $\de$ is good precisely when it is good for every (complete, countable) $T$. 
For more on this correpondence, see \cite{MiSh:1030} \S 2.  

In this paper a main object is to show certain regular ultrafilters are good for some of the $T_{k+1,k}$ while not for others. 
Our first point of leverage for seeing gradations in goodness will be from \cite{MiSh:999}. 
Towards this, let us set notation for Boolean algebras arising as the completion of the Boolean algebra generated by $\alpha$ (usually, $2^\lambda$) 
independent partitions of size $\mu$. 

\begin{defn} 
For an infinite cardinal $\mu$ and an ordinal $\alpha$, 
\begin{enumerate}
\item Let $\fin_{\mu}(\alpha)$ denote the set of partial functions from $\alpha$ to $\mu$ with finite domain. 
\item $\ba^0 = \ba^0_{\alpha, \mu}$ is the Boolean algebra generated by:
\\ $\{ \mx_f : f \in \fin_{\mu}(\alpha) \}$ freely subject to the conditions that
\begin{enumerate}
\item $\mx_{f_1} \leq \mx_{f_2}$ when $f_1 \subseteq f_2 \in \fin_{\mu}(\alpha)$ 
\item $\mx_f \cap \mx_{f^\prime} \neq 0$ iff $f, f^\prime$ are compatible functions. 
\end{enumerate}
\item $\ba^1_{\alpha, \mu}$ is the completion of $\ba^0_{\alpha, \mu}$. 
\item When $\ba$ is a Boolean algebra, $\ba^+$ denotes $\ba \setminus \{ 0 \}$. 
\end{enumerate}
\end{defn}

\begin{conv}
We will assume that giving $\ba$ determines $\alpha$, $\mu$, and a set of generators $\langle \mx_f : f \in \fin_{\mu}(\alpha)\rangle$. 
\end{conv}

\begin{fact} \label{fact-iff}
The existence of $\ba^0_{2^\lambda, \mu}$ thus its completion 
is by Engelking-Karlowicz \cite{ek}.  See also Fichtenholz and Kantorovich\cite{f-k}, 
Hausdorff \cite{hausdorff}, or Shelah \cite{Sh:c} Appendix, Theorem 1.5. 
\end{fact}
The next definition was used in Theorem \ref{t:separation}, i.e. \cite{MiSh:999} Theorem 6.13. 
It allows us to involve arbitrary ultrafilters $\de_*$ on complete Boolean algebras in the construction of regular ultrafilters $\de$.  
\begin{defn}[Regular ultrafilters built from tuples, from \cite{MiSh:999} Theorem 6.13] \label{d:built}
Suppose $\de$ is a regular ultrafilter on $I$, $|I| = \lambda$. We say that $\de$ is built from 
$(\de_0, \ba, \de_*)$ when: 

\begin{enumerate}
\item {$\de_0$ is a regular, $|I|^+$-excellent filter on $I$} 
\\ {$($for the purposes of this paper, it is sufficient to use regular and good$)$}
\item {$\ba$ is a Boolean algebra}
\item {$\de_*$ is an ultrafilter on $\ba$}
\item {there exists a surjective homomorphism $\jj : \mcp(I) \rightarrow \ba$ such that:}
\begin{enumerate}
\item $\de_0 = \jj^{-1}(\{ 1_\ba \})$ 
\item $\de = \{ A \subseteq I : \jj(A) \in \de_* \}$.
\end{enumerate}
\end{enumerate}
\end{defn}
It was verified in \cite{MiSh:999} Theorem 8.1 that whenever $\mu \leq \lambda$ and $\ba = \ba^1_{2^\lambda, \mu}$ 
there exists a regular good $\de_0$ on $\lambda$ and a surjective homorphism $\jj: \mcp(I) \rightarrow \ba$ 
such that $\de_0 = \jj^{-1}(1)$. Thus, Definition \ref{d:built} is meaningful. 

Suppose now that $\de$ is built from $(\de_0, \ba, \de_*)$, witnessed by $\jj$.  Consider a complete countable $T$ and $M \models T$. 
Suppose $N \preceq M^\lambda/\de$, $|N| = \lambda$ and $p \in \ts(N)$, where $p = \langle \vp_\alpha(x,{a}_\alpha) : \alpha < \lambda \rangle$. 
(As ultraproducts commute with reducts, we may assume without loss of generality that $T = T^{eq}$ and so that each $a_\alpha$ is a singleton.) 
For each finite $u \subseteq \lambda$, the \los map \L ~sends $u \mapsto B_u$ where 
\[ B_u := \{ t \in I : M \models (\exists x) \bigwedge_{\alpha \in u} \{ R(x, a_\alpha[t]) \}. \]
Let $\mb_u = \jj(B_u)$.  The key model-theoretic property of the sequence ${\langle \mb_u : u \in [\lambda]^{<\aleph_0} \rangle}$ 
in $\ba$ is captured by the following definition. 

\begin{defn} \label{d:poss} \emph{(Possibility patterns \cite{MiSh:999} Definition 6.1)}
Let $\ba$ be a Boolean algebra, normally complete, and $\bar{\vp} = \langle \vp_\alpha : \alpha < \lambda \rangle$ a sequence of formulas. 
Say that $\overline{\mb}$ is a $(\lambda, \ba, T, \bar{\vp})$-possibility when:
\begin{enumerate}
\item $\overline{\mb} = \langle \mb_u : u \in \lao \rangle$ is a sequence of elements of $\ba^+$
\item if $v \subseteq u \in \lao$ then $\mb_u \subseteq \mb_v$ 
\item if $u_* \in \lao$ and $\bc \in \ba^+$ satisfies
\[ \left( u \subseteq u_* \implies \left( ( \bc \leq \mb_u )  ~\lor~ ( \bc \leq 1 - \mb_u ) \right) \right) \land 
\left( \alpha \in u_* \implies \bc \leq \mb_{\{\alpha\}} \right) \]
then we can find a model $M \models T$ and $a_\alpha \in M$ for $\alpha \in u_*$ such that for every $u \subseteq u_*$,
\[ M \models (\exists x)\bigwedge_{\alpha \in u} \vp_\alpha(x;a_\alpha) ~~ \mbox{iff} ~~ \bc \leq \mb_u  .\] 
\end{enumerate}
If $\Delta$ is any set of formulas, we say $\overline{\mb}$ is a $(\lambda, \ba, T, \Delta)$-possibility if it is a 
$(\lambda, \ba, T, \bar{\vp})$-possibility for some sequence $\bar{\vp}$ of formulas from $\Delta$. 
\end{defn}

In some sense, \ref{d:poss} says that the ``Venn diagram'' of the elements of $\bar{\mb}$ accurately reflects the intersection patterns of the given sequence 
of formulas in the monster model. 
We will often keep track of a full type $p \in \ts(N)$, but 
recall that it suffices to deal with $\vp$-types for each $\vp$
because saturation of ultrapowers reduces to saturation of $\vp$-types, \cite{mm1} Theorem 12. 
 
\begin{defn} \emph{(Moral ultrafilters on Boolean algebras, \cite{MiSh:999} Definition 6.3)} \label{d:moral}
We say that an ultrafilter $\de_*$ on the Boolean algebra $\ba$ is $(\lambda, \ba, T, \Delta)$-moral when
for every $(\lambda, \ba, T, \Delta)$-possibility $\overline{\mb} = \langle \mb_u : u \in \lao \rangle$ 
which is a sequence of elements of $\de_*$, 
there is a multiplicative $\de_*$-refinement $\overline{\mb}^\prime = \langle \mb^\prime_u : u \in \lao \rangle$, i.e.
\begin{enumerate}
\item $u_1, u_2 \in \lao \implies \mb^\prime_{u_1} \cap \mb^\prime_{u_2} = \mb^\prime_{u_1 \cup u_2}$
\item $u \in \lao \implies \mb^\prime_u \subseteq \mb_u$
\item $u \in \lao \implies \mb^\prime_u \in \de_*$. 
\end{enumerate}
We write $(\lambda, \ba, T)$-moral in the case where $\Delta$ is all formulas of the language. 
\end{defn}

The following key theorem of \cite{MiSh:999} connects ``morality'' of $\de_*$ to goodness of $\de$ in the natural way. 

\begin{thm-lit}[``Separation of variables'', Malliaris and Shelah \cite{MiSh:999} Theorem 5.11] \label{t:separation}
Suppose that $\de$ is a regular ultrafilter on $I$ built from $(\de_0, \ba, \de_*)$. Then the following are equivalent: 
\begin{itemize}
\item[(A)] $\de_*$ is $(|I|, \ba, T)$-moral. 
\item[(B)] $\de$ is good for $T$. 
\end{itemize}
\end{thm-lit}

Theorem \ref{t:separation} helps with analyzing the intermediate classes in Keisler's order, as shown in  \cite{MiSh:999}.  
It also focuses the regular ultrafilter construction problems essential to Keisler's order on to the 
problem of constructing ultrafilters $\de_*$ on complete Boolean algebras, where one has a priori 
much more freedom and is not bound by regularity.  Much recent work has focused on building such $\de_*$. 
In the paper \cite{MiSh:1030}, which is foundational for the present argument, we built a powerful family of so-called  
\emph{optimal} ultrafilters over any suitable tuple of cardinals $(\lambda, \mu, \theta, \sigma)$, along with their simpler 
avatars the \emph{perfect} ultrafilters.   
In the present paper, we use $\theta = \sigma = \aleph_0$ so the criterion 
of ``suitable'' reduces to requiring that $\lambda > \mu \geq \aleph_0$.  Given the transparency of the theories involved, 
we have written the present proof to use only the more easily quotable definition of ``perfect.'' 

\begin{defn} \label{d:support}
Let $\overline{\mb} = \langle \mb_u : u \in [\lambda]^{<\aleph_0} \rangle$ be a sequence of elements of $\ba = \ba^1_{\alpha, \mu}$.  
We say $X$ is a support of $\overline{\mb}$ in $\ba$ when $X \subseteq \{ \mx_f : f \in \fin_{\mu}(\alpha) \}$ and 
for each $u \in [\lambda]^{<\aleph_0}$ there is a maximal antichain of $\ba$
consisting of elements of $X$ all of which are either $\leq \mb_u$ or $\leq 1 -\mb_u$. When a support $\supp(\bar{\mb})$ is given, 
write 
\[ \ba^+_{\supp(\bar{\mb}), \mu} \mbox{ to mean } \ba^+_{\alpha_*, \mu} \] 
where $\alpha_* \leq \alpha$ is minimal such that $\bigcup \{ \dom(f) : \mx_f \in \supp(\overline{\mb}) \} \subseteq \alpha_*$. 
\end{defn}

\begin{defn}[Perfect ultrafilters, \cite{MiSh:1030} Definition 9.15\footnote{
Corrected in arxiv v2 to say explicitly that any large enough $\alpha$ will work, matching \cite{MiSh:1030}. This had been the 
intention, so there is no change in the proof.}] 
\label{d:perfect} 
We say that an ultrafilter $\de_*$ on $\ba = \ba^1_{2^\lambda, \mu}$ 
is \emph{$(\lambda, \mu)$-perfect} when $(A)$ implies $(B)$:
\begin{enumerate}
\item[(A)] $\langle \mb_u : u \in [\lambda]^{<\aleph_0} \rangle$ is a monotonic sequence of elements of $\de_*$ 
\\ and  
$\supp(\bar{\mb})$ is a support for $\bar{\mb}$ of cardinality $\leq \lambda$, see $\ref{d:support}$, such that \\ for every $\alpha < 2^\lambda$ with 
$\bigcup \{ \dom(f) : \mx_f \in \supp(\overline{\mb}) \} \subseteq \alpha$, 
\\ there exists a multiplicative sequence 
\[ \langle \mb^\prime_u : u \in [\lambda]^{<\aleph_0} \rangle \]
of elements of $\ba^+$ 
such that
\begin{itemize} 
\item[(a)] $\mb^\prime_{u} \leq \mb_{u}$ for all $u \in [\lambda]^{<\aleph_0}$,  
\item[(b)] for every $\mc \in \ba^+_{\alpha, \mu} \cap \de_*$, 
no intersection of finitely many members of  
$\{ \mb^\prime_{\{i\}} \cup (1-\mb_{\{i\}}) : i < \lambda \}$ 
is disjoint to $\mc$. 
\end{itemize}
\item[(B)] there is a multiplicative sequence $\bar{\mb}^\prime = \langle \mb^\prime_u : u \in [\lambda]^{<\aleph_0} \rangle$ 
of elements of $\de_*$ which refines $\bar{\mb}$. 
\end{enumerate}
\end{defn}

%
%




\begin{defn} \label{d:perfected}
Suppose $\de$ built from $(\de_0, \ba, \de_*)$ where $\de_0$ is a regular filter on $I$, $|I| = \lambda$, 
$\ba = \ba^1_{2^\lambda, \mu}$ and $\de_*$ is $(\lambda, \mu)$-perfect.  In this case we say $\de$ 
is \emph{$(\lambda, \mu)$-perfected.}
\end{defn}

\begin{thm-lit}[\cite{MiSh:1030} Theorem 9.18 and Conclusion 9.20] \label{t:exists} 
For any infinite $\lambda > \mu$, there exists a regular, 
$(\lambda, \mu)$-perfect ultrafilter on $\ba^1_{2^\lambda, \mu}$. 
Moreover, there exists a $(\lambda, \mu)$-perfected, thus regular, ultrafilter on 
$\lambda$ which is not good for any non-simple theory, 
in fact, not $\mu^{++}$-good for any non-simple theory. 
\end{thm-lit}

\br

\section{The saturation condition}


In this section we prove that whenever $n,k, \lambda, \mu$ are such that the property $\pr_{n,k}(\lambda, \mu)$ from \ref{x2} above holds, 
then any $(\lambda, \mu)$ perfect(ed) ultrafilter will be able to handle the theory $T_{n,k}$. 

\br

\begin{theorem} \label{key-sat-1}
Suppse we are given $k<n$, $\mu < \lambda$, $\de$, and $T$, where: 
\begin{enumerate}
\item $\lambda, \mu, n, k$ are such that $\pr_{n,k}(\lambda, \mu)$ holds. 
\item $T = T_{n,k}$. 
\item $\de$ is a $(\lambda, \mu)$-perfected ultrafilter on $I$, $|I| = \lambda$.
\end{enumerate}
Then $\de$ is good for $T$, i.e. for any $M \models T$, $M^I/\de$ is $\lambda^+$-saturated.  
\end{theorem}
\setcounter{equation}{0}
\setcounter{thebigstep}{4}
\setcounter{thestep}{6}

\begin{proof}  
To begin let us fix several objects.
\begin{itemize}
\item The assumption on $\de$ means we may fix $\de_0$, $\ba = \ba^1_{2^\lambda, \mu}$, $\jj$ and a $(\lambda, \mu)$-perfect ultrafilter $\de_*$ on $\ba$ 
such that $\de$ is built from $(\de_0, \ba, \de_*)$ via $\jj$.
\item The fact that $\de$ is regular means we may choose any $M \models T_{n,k}$ as the index model. For convenience, suppose $|M| > \lambda$. 
\item Fix a lifting from $M^I/\de$ to $M^I$, so that for each $a \in M^I/\de$ and each index $t \in I$ the projection $a[t]$ is well defined.  
If $\bar{c}  = \langle c_i : i < m \rangle \in {^m(M^I/\de)}$ then we use $\bar{c}[t]$ to denote $\langle c_i[t] : i < m \rangle$. 
\item Fix a partial type $p = p(x)$ over $A \subseteq M^I/\de$, $|A| \leq \lambda$ which we will try to realize.  
Without loss of generality, $p$ is nonalgebraic 
and $M^I/\de \rstr A \prec M^I/\de$. 
Then, by our choice of theory, it suffices to consider $p \in \ts_\Delta(A)$, $|A| \leq \lambda$ where $\Delta = \{ R(x,x_1,\dots, x_k), \neg R(x,x_1,\dots, x_k) \}$. 
\end{itemize}
With these objects in hand let us proceed with the analysis. 
\begin{equation}
\label{eq-num1a}
\mbox{Let $\langle a_i : i < \lambda \rangle$ list the elements of $A$ without repetition.}
\end{equation}
\begin{equation}
\label{eq-ind2a}
\mbox{Let $\langle v_\beta : \beta < \lambda \rangle$  enumerate $[\lambda]^k$ without repetition.} 
\end{equation}
We will generally use $u,v,w \subseteq \lambda$ for sets of $i$'s and $s \subseteq \lambda$ for sets of $\beta$'s.  
When $w \in {^n \lambda}$ is a finite sequence or a finite set  
(which, for this purpose, we consider as a sequence, in increasing order) 
let $\bar{a}_w$ mean $\langle a_i : i \in w \rangle$. 
Then for some function $\trv: \lambda \rightarrow \{ 0, 1 \}$, (\ref{eq-ind2a}) induces an enumeration of $p$ as 
\begin{equation}
\label{enum-of-p}
 p = \langle R(x, \bar{a}_{v_\beta})^{\trv(\beta)} : \beta < \lambda \rangle. 
\end{equation}
Recall that here $\vp^0 = \neg \vp$, $\vp^1 = \vp$. 
For each $s \in \Omega := [\lambda]^{<\aleph_0}$, we will denote the 
set of indices for vertices appearing in  
$\{ R(x, \bar{a}_{v_\beta})^{\trv(\beta)} : \beta \in s \}$ as follows:
\begin{equation}
 \ver(s) = \bigcup \{ v_\beta : \beta \in s \}.
\end{equation}
In the other direction, let the index operator $\ndx$ accept a finite set of indices for elements of $A$ 
and return the relevant indices for formulas in the type:   
\begin{equation}
 \ndx(u) = \{ \beta < \lambda : v \in [u]^k \mbox{ and } v_\beta = v \}. 
\end{equation}
As we assumed the list (\ref{eq-ind2a}) was 
without repetition, $\ndx: [\lambda]^{<\aleph_0} \rightarrow [\lambda]^{<\aleph_0}$.

Now for each $s \in \Omega$, the \los map \L$: \Omega \rightarrow \de$ sends $s \mapsto B_s$ where 
\begin{equation} \label{by-the-los-map}
B_s := \{ t \in I : M \models (\exists x) \bigwedge_{\beta \in s} R(x, \bar{a}_{v_\beta}[t])^{\trv(\beta)} \}. 
\end{equation}
Define $\mb_s = \jj(B_s)$. We now have a possibility pattern for $T = T_{n,k}$, see Defn. \ref{d:poss}:  
\begin{equation}
\bar{\mb} = \langle \mb_s : s \in \Omega \rangle. 
\end{equation}
With this setup, the strategy for the remainder of the proof will be to construct a sequence $\langle \mb^\prime_s : s \in \Omega \rangle$ 
which, along with $\bar{\mb}$, satisfies the hypotheses of Definition \ref{d:perfect}(A). Then \ref{d:perfect}(B) will guarantee that 
$\bar{\mb}$ has a multiplicative refinement in $\ba$ and thus, by separation of variables,
that $\de$ is good for $T$.  
We will proceed as follows. First, we build an appropriate support for $\bar{\mb}$. Second, we use this data to define associated 
equivalence relations. Third, we define the sequence $\bar{\mb}^\prime$. It will be immediate from the definition that 
this sequence is multiplicative and refines $\bar{\mb}$ on singletons. Fourth, we show that the sequence $\bar{\mb}^\prime$ is not 
trivial, i.e. it satisfies \ref{d:perfect}(A)(b). Finally, we show that $\bar{\mb}^\prime$ is a refinement of $\bar{\mb}$, and thus 
satisfies \ref{d:perfect}(A)(a). 

\br
Our first task is to choose an appropriate support for $\bar{\mb}$ in the sense of \ref{d:support}. 
Following an idea from \cite{MiSh:1009}, whenever $i,j \in \lambda$ let 
\begin{equation}
\label{c-equal}
A_{a_i=a_j} := \{ t \in I : a_i[t] = a_j[t] \} \mbox{ and let } \ma_{a_i = a_j} := \jj(A_{a_i=a_j}). 
\end{equation}
For each $i < \lambda$ let $\eff_{\{i\}}$ be the set of all $f \in \fin_\mu(2^\lambda)$ such that for some $j \leq i$, both (\ref{c:a1}) and (\ref{c:a2}) hold: 
\begin{equation}
\label{c:a1}
\mx_f \leq  \ma_{a_i=a_j}.
\end{equation} 
\begin{equation}
\label{c:a2} \mbox{for all $k<j$, } ~\mx_f \cap \ma_{a_i=a_k} ~ = 0.
\end{equation} 
For each finite $u \subseteq \lambda$, define $\eff_u$ to be $\bigcap \{ \eff_{\{i\}} : i \in u \}$. 
Note that each $\eff_u$ is upward closed, i.e. $f \in \eff_u$ and $g \supseteq f$ implies $g \in \eff_u$. 
For each $s \in \Omega$, the benefit of working with elements of $\eff_{\ver(s)}$ 
will be that we may consider the partial function $i \mapsto \rho_{i}(f)$ on $\fin_\mu(2^\lambda)$ where 
\begin{equation}
\label{defn-of-rho-i}
   \rho_{i}(f) = \operatorname{min} \{ j \leq i : \mx_f \leq \ma_{a_i=a_j} \}. 
\end{equation}
The key point is that if $f \in \eff_{\ver(s)}$ and 
$i \in \ver(s)$ and 
$\rho_i(f) = j$, then for no $f^\prime \supseteq f$ does there exist $j^\prime < j$ such that $\mx_{f^\prime} \leq \ma_{a_i=a_{j^\prime}}$.  

As a result, for each choice of $s \in \Omega$ and 
$f \in \eff_{\ver(s)}$ we may naturally collect 
all the ``active'' indices by mapping 
\begin{equation} 
\label{defn:w}
(s,f) \mapsto w_{s,f} := \{ j \leq i : \mbox{ for some } i \in \ver(s), \mbox{we have } \rho_i(f) = j \} \cup \ver(s). 
\end{equation}
The map (\ref{defn:w}) is really like a finite closure operator: 
for each $s \in \Omega$ and $f \in \eff_{\ver(s)}$, we have that $\ver(s) \subseteq w_{s,\zeta} \in [\lambda]^{<\aleph_0}$, 
$f \in \eff_{w_{s,f}}$, and $w_{\ndx(w_{s,f}), f} = w_{s,f}$. 
Moreover, if $f \in \eff_{\ver(s)}$ then $f \in \eff_{w_{s,f}}$. Notice also that 
\begin{equation}
\label{eq:notice}
\mbox{ for any $s \in \Omega$ and any $\mc \in \ba^+$, there is $f \in \eff_{\ver(s)}$ with $\mx_f \leq \mc$. }
\end{equation} 
Why?   Recall that for $\ma, \mc \in \ba^+$, we say that $\mc$ supports $\ma$ when either $\mc \leq \ma$ or $\mc \leq 1 - \ma$.  
Without loss of generality, $\mc$ supports $\mb_s$.  
Since $\ver(s)$ is finite, it will suffice to prove that for a given $i \in \ver(s)$ we can find 
$f$ such that $\mx_f \leq \mc$ and $f \in \eff_{\{i\}}$. As the generators are dense in the completion, there is 
$f \in \fin_\mu(2^\lambda)$ with $\mx_f \leq \mc$, and (\ref{c:a1}) trivially holds of $f$ in the case $j=i$.   
If (\ref{c:a2}) does not hold in the case $j=i$, there are $i_{1} < i$ and 
$f_{1} \supseteq f$ such that (\ref{c:a1}) holds of $f_{1}$ in the case $j = i_{1}$. Since the ordinals are well ordered, 
after iterating this for finitely many steps we find $j=i_{k}$ and $f_{k} \supseteq \cdots \supseteq f_1 \supseteq f$ for which (\ref{c:a2}) 
also holds.  This proves (\ref{eq:notice}). 

\br

We need one more ingredient to construct the support: the partitions should decide not only equality but also the formulas $R$ on elements from $w_{s,\zeta}$. 
Towards this, 
for each $u \in [\lambda]^{k+1}$, write 
\begin{equation}
\label{eq:r2}
 \ma_{R(\bar{a}_u)} = \jj ( ~ \{ t \in I : M \models R(\bar{a}_u)  \} ).
\end{equation}
We may also say that $1 - \ma_{R(x,\bar{a}_v)} = \ma_{\neg R(x,\bar{a}_v)}$ and $1 - \ma_{R(\bar{a}_u)} = \ma_{\neg R(\bar{a}_u)}$, 
naturally defined. 
We may now state a definition. There is a component of support and a component of coherence across all $s \in \Omega$.  
\begin{equation} 
\label{d:good-support}
\mbox{ $\bar{f} = \langle~ \bar{f}_s = \langle (\fdp_{s, \zeta}, w_{s,\zeta}) : \zeta < \mu \rangle~ : s \in \Omega \rangle$ is a \emph{good support} for $\bar{\mb}$ when:} 
\end{equation}
\begin{enumerate}
\item[(1)] for each $s \in \Omega$,
\begin{enumerate}
\item[(a)] for each $\zeta < \mu$, $f = f_{s,\zeta} \in \eff_{\ver(s)}$. 
\item[(b)] for each $\zeta < \mu$, $w_{s,\zeta} = w_{s, f_{s,\zeta}}$, which is well defined by (a).   
\item[(c)] the sequence $\langle \mx_{f_{s,\zeta}} : \zeta < \mu \rangle$ 
is a maximal antichain of $\ba$ supporting each element of the set\footnote{By condition (1)(a), $\{ \ma_{a_i = a_j} : i, j \in w_{s,\zeta} \}$ are implicitly also here.} 
\[ \{ \mb_{s^\prime} : s^\prime \subseteq s \} 
\cup \{ \ma_{R(\bar{a}_u)} : u \in [w_{s,\zeta}]^{k+1} \}. \] 
\end{enumerate}
\item[(2)] for each $s, s^\prime \in \Omega$ with $s^\prime \subseteq s$, $\bar{f}_{s}$ refines $\bar{f}_{s^\prime}$. 
\item[(3)] for every finite $X \subseteq \bigcup \{ \dom(\fdp_{s,\zeta}) : s \in \Omega, \zeta < \mu \}$ 
and every $s \in \Omega$, there is $s_* \in \Omega$ such that $s \subseteq s_*$ and 
$\zeta < \mu \implies X \subseteq \dom(\fdp_{s_*, \zeta})$. 
\end{enumerate} 
One way of building a good partition is to miniaturize the argument from \cite{MiSh:1030}, as follows. 
First, we address (1)(a)+(c). For each $s \in \Omega$, we try to choose $f_{s,\zeta}$ by induction on $\zeta < \mu^+$ such that $0 \in \dom(f_{s,\zeta})$.  
Arriving to $\zeta$, suppose we have some remaining unallocated $\mc \in \ba^+$, i.e. a nonzero $\mc$ disjoint to $\bigcup \{ \mx_{f_{s,\gamma}} : \gamma < \zeta \}$. 
Without loss of generality, $\mc$ supports $\mb_s$.  By (\ref{eq:notice}), we may choose $f \in \eff_{\ver(s)}$ 
so that $\mx_f \leq \mc$. 
Condition (1)(c) asks that $\mx_f$ also support each element of a finite set, so without loss of generality 
(by taking intersections) we may assume (c) is satisfied. This completes the inductive step. 
As no antichain of $\ba$ has cardinality greater than $\mu$, the construction will stop at an ordinal $<\mu^+$, but 
as $0 \in \dom(f_{s,\zeta})$ for each $\zeta$ the ordinal is $\geq \mu$. Without loss of generality 
the sequence is indexed by $\mu$. 
Then (1)(b) holds by (\ref{defn:w}). 

To ensure conditions (2) and (3), we refine the partitions just obtained. Let $\langle s_\ell : \ell < \lambda \rangle$ list $\Omega$. 
We update $\overline{f}_{s_\ell} = \langle \fdp_{s_\ell,\zeta} : \zeta < \mu \rangle$ by induction on $\ell <\mu$ as follows. 
Arriving to $\ell$, if $(\exists k < \ell)(s_\ell \subseteq s_k)$ then let $k(\ell) = \min \{ k<\ell : s_\ell \subseteq s_k \}$ and let 
$\overline{f}_{s_\ell} = \overline{f}_{s_k}$. If there is no such $j$, we choose $\overline{f}_{s_\ell}$ such that it refines 
$\overline{f}_{s_k}$ (i.e. every $\fdp_{s_\ell,\zeta}$ extends $\fdp_{s_k,\zeta}$ for some $\zeta < \mu$) 
whenever $k < \ell$ and $s_k \subseteq s_\ell$. There are at most $2^{|s_\ell|} < \aleph_0$ such $j$ so this can be done.

At the end of this process, if necessary, we may re-index the partitions so that they are of order type $\mu$.   
By construction, for each $s \in \Omega$ and $\zeta < \mu$ the set $w_{s,\zeta}$ is well defined by (1)(b). This completes the construction 
of a good support for $\bar{\mb}$. 
\begin{equation}
\label{choice-of-f} 
\mbox{ For the remainder of the proof, we fix a good support $\bar{f}$ for $\bar{\mb}$. }
\end{equation}
\begin{equation} 
\label{defn-mcv}
\mbox{Fix $\mcv \subseteq 2^\lambda$, $|\mcv| \leq \lambda$ such that }\bigcup \{ \dom(f_{u,\zeta}) : u \in \Omega, \zeta < \mu \} \subseteq \mcv.
\end{equation}
Finally, for each $s \in \Omega$ and each $\zeta < \mu$, define\footnote{Informally, elements of $\gee_{s,\zeta}$ specify consistent $R$-types over the parameters with indices in $\ndx(w_{s,\zeta})$. Edges only hold on distinct tuples since $R$ is irreflexive. Given two tuples which ``collapse'' to the same values, either both or neither have an edge. 
The type extends $p \rstr s$ if possible, that is, if the \los map allows it.  
In the case where $\jj$ is the identity so the 
elements $\mx_{f_{s,\zeta}}$ are subsets of $I$, the reader may think of $g_{s,\zeta}$ as coding an $R$-type over $\{ a_i[t] : i \in w_{s,\zeta} \}$ 
which is consistent for any $t \in w_{s,\zeta}$. We will essentially arrive at this picture towards the end of the proof; we will find a set 
$C$ such that (among other things) $\jj(C) \subseteq \mx_{f_{s,\zeta}}$, choose $t \in C$ and consider the type given by $g_{s,\zeta}$ at $t$.}
\begin{equation}
\label{d:float}
\mbox{ 
$\gee_{s,\zeta}$ to be the set of functions $g = g_{s,\zeta}: \ndx(w_{s,\zeta}) \rightarrow \{ 0, 1 \}$ such that:} 
\end{equation}
\begin{enumerate}
\item[(a)] if $\mx_{f_{s,\zeta}} \leq \mb_s$, then for all $\beta \in s$, $g_{s,\zeta}(\beta) = \trv(\beta)$. 
\item[(b)] if $\gamma \in \ndx(w_{s,\zeta})$, $i \neq j \in v_\gamma$ and $\rho_i(f_{s,\zeta}) = \rho_j(f_{s,\zeta})$, then $g_{s,\zeta}(\gamma) = 0$. 
\item[(c)] if $\beta \neq \gamma \in \ndx(w_{s,\zeta})$ and 
$\{ \rho_i(f_{s,\zeta}) : i \in v_\gamma \} = \{ \rho_i(f_{s,\zeta}) : i \in v_\beta \}$ then $g_{s,\zeta}(\gamma) = g_{s,\zeta}(\beta)$. 
\item[(d)] if $w \in [w_{s,\zeta}]^n$ and $\mx_{f_{s,\zeta}} \leq \ma_{R(\bar{a}_u)}$
for each $u \in [w]^{k+1}$, then $g_{s,\zeta} \rstr \ndx(w)$ is not constantly 1.    
\end{enumerate}
Regarding condition (a), recall that by construction in (\ref{d:good-support})(1)(c) $\mx_{f_{s,\zeta}}$ decides $\mb_{s^\prime}$ for all $s^\prime \subseteq s$ and 
it likewise decides $\ma_{R(\bar{a}_u)}$ for each $u \in [w_{s,\zeta}]^{k+1}$. 
To see that $\gee_{s,\zeta} \neq \emptyset$ simply involves unwinding the definition. There are two cases. If $\mx_{f_{s,\zeta}} \leq 1-\mb_s$ then 
$\gee_{s,\zeta}$ contains the function which is constantly $0$. If $\mx_{f_{s,\zeta}} \leq \mb_s$, then 
recalling (\ref{eq:r2}) we have that if there is $u \in [w_{s,\zeta}]^{k+1}$ such that each $v \in [u]^k$ is $v_\beta$ for some 
$\beta \in s$, then $\mx_{f_{s,\zeta}} \leq 1 - \ma_{R(\bar{a}_u)}$. Thus, we may 
set $g(\gamma) = 1$ if and only if $\gamma \in \ndx(w_{s,\zeta})$ and 
$\{ \rho_i(f_{s,\zeta}) : i \in v_\gamma \} = \{ \rho_i(f_{s,\zeta}) : i \in v_\beta \}$ for some $\beta \in s$. 
In other words, since each $\rho_i$ is $s$-accurate, 
it is sufficient to give the behavior of $g$ on the set 
$\{ \rho_i(f_{s,\zeta}) : i \in w_{s,\zeta} \}$, as the condition of $s$-accurate and the definition (\ref{eq:r2}) 
ensure that if $u, u^\prime \in [w_{s,\zeta}]^{k+1}$ and 
$\{ \rho_i(f_{s,\zeta}) : i \in u \} = \{ \rho_i(f_{s,\zeta}) : i \in u^\prime \}$ then 
$\mx_{f_{s,\zeta}} \leq \ma_{R(\bar{a}_u)}$ if and only if $\mx_{f_{s,\zeta}} \leq \ma_{R(\bar{a}_{u^\prime})}$. 
So indeed $\gee_{s,\zeta} \neq \emptyset$. For the remainder of the proof, 
\begin{equation}
\mbox{ for each $s \in \Omega$ and $\zeta < \mu$, fix $g_{s,\zeta} \in \gee_{s,\zeta}$. }
\end{equation}
We will informally refer to these objects $g_{s,\zeta}$ as ``floating types.''




\br

Our second task is to organize the data already obtained in terms of a family of equivalence relations.  
This will elide some of the background noise and so give us a cleaner picture of any barriers to realizing the type.  
By hypothesis (1) of the Theorem, $\pr_{n,k}(\lambda, \mu)$ holds. Thus, identifying $\lambda$ with the set of indices for 
elements of $A$ as in (\ref{eq-num1a}), let us fix $G: [\lambda]^{<\aleph_0} \rightarrow \mu$ such that:   
\begin{quotation}
for each $w \in [\lambda]^n$ and each sequence $\langle u_v : v \in [w]^k \rangle$ of finite subsets of $\lambda$
such that $v \in [w]^k$ implies $v \subseteq u_v$   
and $G \rstr \{ u_v : v \in [w]^{k} \}$ is constant, there is $v \in [w]^{k}$ such that $w \subseteq u_v$. 
\end{quotation}
Let $E$ be the equivalence relation on $W = \Omega \times \mu \times \mu$ given by: 
\begin{equation}
\label{the-eq-reln}
E((s, \zeta, \xi),(s^\prime, \zeta^\prime, \xi^\prime)) \mbox{ if and only if } 
\end{equation}
\begin{enumerate}[label={(\alph*)}]
\item $\zeta = \zeta^\prime$ and $\xi = \xi^\prime$. 
\item $\otp(s) = \otp(s^\prime)$, $\otp(\ndx(w_{s,\zeta})) = \otp(\ndx(w_{s^\prime, \zeta}))$ and the order preserving map 
from $\ndx(w_{s,\zeta})$ onto $\ndx(w_{s^\prime, \zeta})$ takes $s$ to $s^\prime$. 
\item $\otp(\ver(s)) = \otp(\ver(s^\prime))$, $\otp(w_{s,\zeta}) = \otp(w_{s^\prime,\zeta})$ and the order preserving map 
from $w_{s,\zeta}$ onto $w_{s^\prime, \zeta}$ takes $\ver(s)$ to $\ver(s^\prime)$. 
\item $\otp(\dom(f_{s,\zeta})) = \otp(\dom(f_{s^\prime, \zeta}))$. 
\item if $\gamma_s \in \dom(f_{s,\zeta})$, $\gamma_{s^\prime} \in \dom(f_{s^\prime, \zeta})$ and 
$\otp(\gamma_s \cap \dom(f_{s,\zeta})) = \otp(\gamma_{s^\prime} \cap \dom(f_{s^\prime, \zeta}))$ then $f_{s,\zeta}(\gamma_s) = f_{s^\prime, \zeta}(\gamma_{s^\prime})$. 
\item $\langle g_{s,\zeta}(\beta) : \beta \in \ndx(w_{s,\zeta}) \rangle = \langle g_{s,\zeta}(\beta) : \beta \in \ndx(w_{s,\zeta}) \rangle$. 
\item $G(w_{s,\zeta}) = G(w_{s^\prime, \zeta}) = \xi$. 
\end{enumerate}
Since the sets and ordinals in question are all finite, but $\zeta < \mu$ may vary, it is easy to see that 
there are precisely $\mu$ equivalence classes of $E$. Choose an enumeration of these classes as 
\begin{equation}
\bar{W} = \langle W_{\epsilon} : \epsilon < \mu \rangle, \mbox{ so $W = \bigcup_\epsilon W_\epsilon$}. 
\end{equation}
Fix a representative function 
\begin{equation} \label{the-function-h}
h: \mu \rightarrow W \mbox{ such that } h(\epsilon) \in W_\epsilon. 
\end{equation}
In the rest of the proof, we will often denote the values of $\zeta, \xi$ at $h(\epsilon)$ by
$\zeta_{h(\epsilon)}, \xi_{h(\epsilon)}$ respectively. 
The next definition will be central. 
For each $\beta < \lambda$, $\epsilon < \mu$ let us collect all elements of $\Omega$ which occur as part of an $\epsilon$-template tuple 
$(s,\zeta, \xi)$ where $\beta \in s$ and $\mx_{\fdp_{s,\zeta}} \leq \mb_{s}$: 
\begin{equation} 
\label{uu-eqn}
\uu_{\beta,\epsilon} = \{ s  \colon (s, \zeta_{h(\epsilon)}, \xi_{h(\epsilon)}) \in W_\epsilon, ~ \beta \in s, ~\mx_{\fdp_{s, \zeta_{h(\epsilon)}}} \leq \mb_{s} \}. 
\end{equation} 
A useful property of these sets is the following: for each $\epsilon < \mu$, 
\begin{equation} 
\label{constant-color}
\mbox{if $s \in \uu_{\beta, \epsilon}$ and $s^\prime \in \uu_{\beta^\prime, \epsilon}$ 
then $G(w_{s,\zeta_{h(\epsilon)}}) = G(w_{s^\prime, \zeta_{h(\epsilon)}}) = \xi_{h(\epsilon)}$. }
\end{equation}
This completes our construction of the equivalence relations. We now have the necessary scaffolding for the third task. 
\br

Our third task is to define the sequence $\bar{\mb}^\prime$. Recalling $\mcv$ from (\ref{defn-mcv}), fix $\alpha < 2^\lambda$ 
so that $\mcv \subseteq \alpha$. Without loss of generality, $\alpha \geq \lambda$. 
We now copy the functions $f_{s,\zeta}$ onto a new domain where new partitions will allow us to code 
additional information.\footnote{Compare the usual construction of good ultrafilters.} Let $\code_m$ denote some fixed one-to-one $m$-fold coding function from $\lambda^m$ to $\lambda$. 
Let $\tv$ denote the truth value of an expression (either $0$ or $1$). 
\begin{equation}
\label{defining-f-star}
\mbox{For each $s \in \Omega$, $\zeta < \mu$ define $f^* = f^*_{s,\zeta}$ as follows.}
\end{equation} 
\begin{enumerate}
\item $\dom(f^*) \subseteq \alpha \cdot 2 + \lambda \cdot 5$ is finite, $\rn(f^*) \subseteq \mu$, and $f^*$ is determined by the remaining conditions. 
\br
\item if $\gamma \in \dom(\fdp_{s,\zeta})$ then 
\[ f^*(\alpha + \gamma) = \code_2(\fdp_{s,\zeta}(\gamma), \otp(\gamma \cap \dom(\fdp_{s,\zeta}))). \]
\item if $\gamma = \langle i, j \rangle \in \rn(\code_2(w_{s,\zeta} \times w_{s,\zeta}))$, then 
\[ f^*(\alpha\cdot 2 + \gamma) = \tv(\rho_i(f_{s,\zeta}) = \rho_j(f_{s,\zeta})). \]
\item if $\gamma = \langle i_1, \dots, i_k \rangle \in \rn(\code_k(w_{s,\zeta} \times \cdots \times w_{s,\zeta}))$, 
then 
\[ f^*(\alpha\cdot 2 + \lambda + \gamma) = \tv( \mx_{f_{s,\zeta}} \leq \mb_s ). \]
\item if $\gamma = \langle i_0, \dots, i_k \rangle \in \rn(\code_{k+1}(w_{s,\zeta} \times \cdots \times w_{s,\zeta}))$, 
then 
\[ f^*(\alpha\cdot 2 + \lambda\cdot 2 + \gamma) = \tv( \mx_{f_{s,\zeta}} \leq \ma_{R(\bar{a}_{\langle i_0, \dots, i_k \rangle})} ). \]
\item if $\gamma \in w_{s,\zeta}$, then 
\[ f^*(\alpha\cdot 2 + \lambda\cdot 3 + \gamma) = \code_3(\tv(\gamma \in \ver(s)), \otp(\gamma \cap \ver(s)), \otp(\gamma \cap w_{s,\zeta})). \]
\item if $\gamma \in \ndx(w_{s,\zeta})$, then 
\[ f^*(\alpha\cdot 2 + \lambda\cdot 4 + \gamma) = \code_4(\tv(\gamma \in s), \otp(\gamma \cap s), \otp(\gamma \cap \ndx(w_{s,\zeta})), g_{s,\zeta}(\gamma)). \]
\end{enumerate} 
\noindent This completes the definition (\ref{defining-f-star}). 
Of course, this definition could be made more efficient and the domain smaller (say, by more judicious use of $\code_m$). 
Finally, 
\begin{equation}
\mbox{let $\bar{\mc} = \langle \mc_\epsilon : \epsilon < \mu \rangle$ be given by $\mc_{\epsilon} = \mx_{\{(\alpha + \alpha + \lambda \cdot 5, \epsilon)\}} \}$.}
\end{equation}
This new antichain will help us to divide the work in the next definition. Notice that any of its elements 
will have nonzero intersection with any of the elements from $\ba^+_{\alpha + \alpha + \lambda \cdot 5}$. 

We have all the ingredients to define $\bar{\mb}^\prime$.
For each $\beta < \lambda$, let
\begin{align} \label{mb-prime-b}
\mb^\prime_{\{\beta\}} = \left(  \bigcup \{ \mc_\epsilon \cap \mx_{f^*_{s, \zeta_{h(\epsilon)}}} \cap
 \mx_{\fdp_{s, \zeta_{h(\epsilon)}}}  ~:~  \epsilon <\mu, ~ s \in \uu_{\beta, \epsilon} \}  \right) \cap \mb_{\{\beta\}}. 
\end{align}
Let us justify that (\ref{mb-prime-b}) is not zero: for each $\epsilon < \mu$ such that 
$\uu_{\beta,\epsilon} \neq \emptyset$, and for each $s \in \uu_{\beta, \epsilon}$,
\[  \mc_\epsilon \cap \mx_{f^*_{s, \zeta_{h(\epsilon)}}} \cap
 \mx_{\fdp_{s, \zeta_{h(\epsilon)}}} \cap \mb_{\{\beta\}} ~ > ~ 0. \]
This is because domains of the functions corresponding to 
$\mx_{\fdp_{s, \zeta_{h(\epsilon)}}}$, $\mc_\epsilon$ and $\mx_{f^*_{s, \zeta_{h(\epsilon)}}}$ are 
mutually disjoint, and adding $\mb_{\{\beta\}}$ is allowed by the definition of $\uu_{\beta,\epsilon}$. 
(Recall that by monotonicity, $\beta \in s$ implies $\mb_s \leq \mb_{\{\beta\}}$.)
For each $s \in \Omega \setminus \emptyset$, define
\begin{align} \label{mb-prime-s}
\mb^\prime_s = \bigcap \{ \mb^\prime_{\{ \beta \}} : \beta \in s \}.  
\end{align}
Let $\mb^\prime_{\emptyset} = 1_{\ba}$. This completes the definition of 
the sequence $\bar{\mb}^\prime$:
\begin{equation}
\label{defn-of-b-prime}
\bar{\mb}^\prime = \langle \mb^\prime_u : u \in \Omega \rangle. 
\end{equation}
By construction, $\bar{\mb}^\prime$ is multiplicative, and $\mb^\prime_s \leq \mb_s$ when $|s| = 1$. 


\br
\br

Our fourth task is to prove that the sequence $\bar{\mb}^\prime$ defined in (\ref{defn-of-b-prime}) satisfies Definition 
\ref{d:perfect}(A)(b) along with $\bar{\mb}$ and the choice of support $\bar{f}$ determined earlier in the proof (i.e. $\alpha_*$ 
of Definition \ref{d:support} may be taken to be the $\alpha$ of the present proof). Compare this to the  
Step 8 Claim of \cite{MiSh:1030} 6.2.

As the generators are dense in the completion, it will suffice to show that for any $f \in \fin_\mu(\alpha)$,  
any finite $\mci \subseteq \lambda$, and any $\ma \in \de_*$ such that $\supp(\ma) \subseteq \alpha$, 
\begin{equation} 
\label{will-suffice-1}
\ma \cap \bigcap \{ \mb^\prime_{\{\beta\}} \cup (1-\mb_{\{\beta\}}) : \beta \in \mci \} > 0. 
\end{equation}
Taking intersections if necessary, we may write $\mci$ as the disjoint union of $\mcin$ and $\mciy$ where for each $\beta \in \mcin$, 
$\ma \leq 1-\mb_{\{\beta\}}$ and for each $\beta \in \mciy$,  
$\ma \leq \mb_{\{\beta\}}$. Recalling that $\mb^\prime_s \leq \mb_s$ when $|s| = 1$, we suppose that $\mciy$ is nonempty 
(otherwise we are done) and it will suffice to show that 
\begin{equation}
\label{it-is-nonzero}
\ma \cap \bigcap \{ \mb^\prime_{\{\beta\}} : \beta \in \mciy \} > 0. 
\end{equation}
As $\mb_{\mciy} \in \de_*$, without loss of generality $\ma \leq \mb_{\mciy}$ and we can find $f \in \fin_\mu(\alpha)$ 
such that $\mx_f \leq \ma$. 
Recall $\mcv$ from (\ref{defn-mcv}).
Write $f$ as the disjoint union $f^{\operatorname{in}} \cup f^{\operatorname{out}}$  
where $\dom(f^{\operatorname{in}}) \subseteq \mcv$ and $\dom(f^{\operatorname{out}}) \subseteq \alpha \setminus \mcv$. 
Necessarily $\mb_{\mciy} \cap \mx_{f^{\operatorname{in}}} > 0$. 
As $\bar{f}_{\mciy}$ gives rise to a partition, 
let $\zeta_* < \mu$ be such that 
\begin{equation}
\label{e:17}
\mx_{\fdp_{\mciy, \zeta_*}} \cap \mx_{f^{\operatorname{in}}} \cap \mb_{\mciy} > 0. 
\end{equation}
Recall the function $G$ which was given as a witness to $\pr$. 
Let $\xi_* = G(w_{\mciy, \zeta_*})$ and let 
$\epsilon < \mu$ be such that $(\mciy, \zeta_*, \xi_*) = (\mciy, \zeta_{h(\epsilon)}, \xi_{h(\epsilon)}) \in W_\epsilon$.   
Going forward, we will write $\zeta_{h(\epsilon)}$ instead of $\zeta_*$ for clarity. As we have $\mx_f \leq \mb_{\mciy}$, 
it follows from the definition (\ref{uu-eqn}) that 
\begin{equation} 
\mciy \in \uu_{\beta,\epsilon} \mbox{ for each } \beta \in \mciy. 
\end{equation}
Now let us verify that   
\begin{equation} \label{e8}
0 < \mx_{f^{\operatorname{out}}} \cap \mx_{f^{\operatorname{in}}} \cap \left( \mc_\epsilon \cap \mx_{f^*_{\mciy, \zeta_{h(\epsilon)}}} \cap
 \mx_{\fdp_{\mciy, \zeta_{h(\epsilon)}}} \right) \cap \mb_{\mciy}.
\end{equation}
The reason is that conflicts can only arise when the domains of the relevant functions intersect. 
By construction, 
\[ \mc_\epsilon, ~\mx_{f^{\operatorname{out}}}, ~\mx_{f^*_{\mciy, \zeta_{h(\epsilon)}}}, ~\mx_{\fdp_{\mciy, \zeta_{h(\epsilon)}}} \]
do not interfere with each other and the first three do not interfere with $\mx_{f^{\operatorname{in}}}$ or with $\mb_{\mciy}$. 
By (\ref{e:17}) $\mx_{\fdp_{\mciy, \zeta_{h(\epsilon)}}} \cap \mx_{f^{\operatorname{in}}} \cap \mb_{\mciy}$ is nonzero.  Replacing $\mx_f = \mx_{f^{\operatorname{in}}} \cap \mx_{f^{\operatorname{out}}}$ 
and quoting the definition of $\mb^\prime_{\mciy}$ in (\ref{mb-prime-b}) and (\ref{mb-prime-s}), we are done. 
This completes the proof of (\ref{will-suffice-1}). 
\br

\br 


\vspace{5mm}
To complete the proof of Theorem \ref{key-sat-1}, it remains to show that for each $s \in \Omega$, $\mb^\prime_s \leq \mb_s$. This will suffice for 
\ref{d:perfect}(A)(a). The background template for our argument is \cite{MiSh:1030} Claim 6.2, Step 10, item (5). 
Before beginning this proof, note that by our definition of the sequence $\bar{\mb}^\prime$, whenever   
$0 < \mc \leq \mc_\epsilon \cap \mb^\prime_{\{\beta\}}$, necessarily 
\begin{align}\label{eb:61}
\bigcup \{ \mc \cap \mx_{f^*_{s,\zeta_{h(\epsilon)}}} \cap \mx_{\fdp_{s, \zeta_{h(\epsilon)}}} :  s \in \uu_{\beta,\epsilon} \} > 0.  
\end{align} 
In particular, under this hypothesis, there is $s \in \uu_{\beta,\epsilon}$ such that 
\begin{equation}
\label{eb:62}
\mc \cap \mx_{f^*_{s,\zeta_{h(\epsilon)}}} \cap \mx_{\fdp_{s, \zeta_{h(\epsilon)}}} > 0 ~~\mbox{\hspace{3mm} thus \hspace{3mm}}~~ 
\mc \cap \mx_{\fdp_{s, \zeta_{h(\epsilon)}}} \cap \mb^\prime_{\{\beta\}} > 0. 
\end{equation}

Now suppose for a contradiction that $\bar{\mb^\prime}$ is {not} a multiplicative refinement of $\bar{\mb}$. 
Then for some finite $\mci \subseteq \lambda$ and some $\mc_0 \in \ba^+$,  
\begin{equation} \label{for-a-contradiction-1}
\mc_0~ \leq ~ \mb^\prime_{\mci} \setminus \mb_{\mci}~ = ~ \bigcap_{\beta \in \mci} \mb^\prime_\beta \setminus \mb_{\mci}. 
\end{equation} 
Without loss of generality, $\mc_0 \leq \mc_\epsilon$ for some $\epsilon < \mu$ and $\mc_0 = \mx_f$ for some $f \in \fin_\mu(2^\lambda)$. 
Enumerate $\mci$ as $\langle \beta_i : i < |\mci| \rangle$.  
Working in $\ba$, by induction on $i < |\mci|$ 
\begin{equation}
\label{f-star-by-ind}
\mbox{ we choose functions $f_{i}$ and sets $s_{\beta_i}$ such that: }
\end{equation}
\begin{enumerate}
\item[(i)] $f_{i} \in \fin_\mu(2^\lambda)$
\item[(ii)] $j < i$ implies $f_{j} \subseteq f_{i}$
\item[(iii)] $s_{\beta_i} \in \uu_{\beta_i, \epsilon}$
\item[(iv)] $ f_{{i}} \supseteq f_{s_{\beta_i}, \zeta_{h(\epsilon)}} \cup f^*_{s_\beta, \zeta_{h(\epsilon)}}.$ 
\end{enumerate}
\noindent Let $f_{-1} = f$. Suppose we have defined $f_j$ for $-1 \leq j < j+1 = i$, and we define 
$f_i$ and $s_{\beta_i}$ as follows. By hypothesis,  
\begin{equation}
\label{ind-x} \mx_{f_j} \leq \mb^\prime_{\mci} \cap \mc_\epsilon. 
\end{equation}
First note that by (\ref{ind-x}) and monotonicity of $\bar{\mb}^\prime$, 
\begin{equation} 
\label{eq54}
\mx_{f_j} \leq \mb^\prime_{\{\beta_i\}} \cap \mc_\epsilon.
\end{equation} 
Second, by (\ref{for-a-contradiction-1}), $\mc_0 \leq \mc_\epsilon \cap  \mb^\prime_{ \{ \beta_i \}}$.
Thus by (\ref{eb:61}),  $\uu_{\beta_i, \epsilon} \neq \emptyset$.  
Apply (\ref{eb:62}) to choose $s_{\beta_i} \in \uu_{\beta_i, \epsilon}$ such that
\[ \mx_{f_j} \cap \mx_{\fdp_{s_{\beta_i}, \zeta_{h(\epsilon)}}} \cap \mx_{f^*_{s_{\beta_i}, \zeta_{h(\epsilon)}}} > 0. \]
Combining this equation with (\ref{eq54}), 
\[ \mx_{f_j} \cap \mc_\epsilon \cap \mb^\prime_{\{\beta_i\}} \cap \mx_{\fdp_{s_{\beta_i}, \zeta_{h(\epsilon)}}} 
\cap \mx_{f^*_{s_{\beta_i}, \zeta_{h(\epsilon)}}} > 0. \]
Let $f_{i} = f_j \cup \fdp_{s_{\beta_i}, \zeta_{h(\epsilon)}} \cup f^*_{s_{\beta_i}, \zeta_{h(\epsilon)}}$. 
This completes the induction. 

\br
\begin{equation}
\mbox{ Let $f_* := \bigcup_{i<|\mci|} f_i$ and let $\langle s_{\beta_i} : i < |\mci| \rangle = \langle s_{\beta} : \beta < \beta_* \rangle$ be as given by this proof. }
\end{equation}
Note that by construction, 
\begin{equation}
\label{smaller}
\mbox{ for each $\beta \in \mci$}, ~\mx_{f_*} \leq \mx_{f_{s_\beta, \zeta_{h(\epsilon)}}}. 
\end{equation}
Consider the set of indices for `active' elements:
\begin{equation}
W = \bigcup \{ w_{s_\beta, \zeta_{h(\epsilon)}} : \beta \in \mci \}. 
\end{equation}
To finish the argument, 
we will move back to the index model. Informally, the point will be that $\mx_{f_*}$ holds open a `space' in the Boolean algebra 
which reflects a particular configuration at some index $t \in I$ (a configuration which we will show cannot happen). 
First, we shall be careful to choose an appropriate $t$, as follows. 
Since the theory $T_{n,k}$ is $\aleph_0$-categorical, let $\Gamma = \Gamma(W)$ be the finite set of formulas
in the variables $\{ x_i : i \in W \}$. For $v \subseteq W$, let $\vp(\bar{x}_v)$ denote that the free variables of $\vp$ 
are among $\langle x_i : i \in v \rangle$, and as above let $\bar{a}_v$ denote $\langle a_i : i \in v \rangle$. 
For each $\vp = \vp(\bar{x}_v) \in \Gamma$, the \los map gives 
\[ C_{\vp(\bar{a}_v)} := \{ t \in I : M \models \vp(\bar{a}_v[t]) \} \mbox{ and let } \mc_{\vp(\bar{a}_v)} = \jj(C_{\vp(\bar{a}_v)}). \]
$\Gamma$ is finite, so we may assume, without loss of generality (by increasing $f_*$ if necessary), 
that $\mx_{f_*}$ supports (decides) each of the finitely many $\mc_{\vp(\bar{a}_v)}$.  More precisely, 
we may assume $\Gamma$ admits a partition into disjoint sets $\Gamma_0 \cup \Gamma_1$ where 
\[ \vp(\bar{x}_v) \in \Gamma_{0} \mbox{ if and only if } \tv\left( \mx_{f_*} \leq \mc_{\vp(\bar{a}_v)} \right) = 0. \]  
The ``accurate'' subset of $I$ is the one defined by 
\[ C := \bigcap \{ C_{\vp(\bar{a}_v)} : \vp(\bar{x}_v) \in \Gamma_1 \} \cap \bigcap \{ I \setminus C_{\vp(\bar{a}_v)} : \vp(\bar{x}_v) \in \Gamma_0 \}~ \subseteq I. \]
Since $\jj(C) \geq \mx_{f_*} > 0$, necessarily $C$ is nonempty. 
\begin{equation}
\mbox{ Fix some $t \in C$ (so $t \in I$) for the remainder of the proof. }
\end{equation}
Now consider the picture in the model $M$ given by index $t$. The set of elements $\{ a_i[t] : i \in W \}$ accurately reflects the picture given by 
$\mx_{f_*}$ in the following ways. First, if $j \leq i \in W$, then $M \models a_i[t] = a_j[t]$ if and only if $\rho_i(f_*) = \rho_j(f_*)$. 
Second, for all $u \in [W]^{k+1}$, $M \models R(\bar{a}_u[t])$ if and only if $\mx_{f_*} \leq \ma_{R(\bar{a}_u)}$ in the sense of (\ref{eq:r2}). 
Moreover, for each $\beta \in \mci$, $\mx_{f_*} \leq \ma_{R(\bar{a}_u)}$ if and only if $\mx_{f_{s_\beta, \zeta_{h(\epsilon)}}} \leq \ma_{R(\bar{a}_u)}$. 

At the given index $t$, the ``floating types'' of (\ref{d:float}) have become actual partial types, which we now name. For each $\beta \in \mci$, let
\begin{equation}
 r_\beta(x) := \{ R(x,\bar{a}_{v_\gamma})^{g_{s_\beta, \zeta_{h(\epsilon)}}(\beta)} : \gamma \in \ndx(w_{s_\beta, \zeta_{h(\epsilon)}}) \}. 
\end{equation}
Condition \ref{d:float}(b) ensures that each $r_\beta(x)$ is a complete, consistent $R$-type over $\{ a_i[t] : i \in w_{s_\beta, \zeta_{h(\epsilon)}} \}$. 
Condition \ref{d:float}(a) ensures that $R(x,\bar{a}_{v_\beta})^{\trv(\beta)} \in r_\beta(x)$, because $\beta \in s_\beta \in \uu_{\beta, \epsilon}$. 
However, $\bigcup \{ r_\beta(x) : \beta \in \mci \}$ is \emph{not} a consistent partial type. This is because something even stronger is true: 
\begin{equation}
\label{d:inc-a}
\{ R(x,\bar{a}_{v_\beta}[t])^{\trv(\beta)} : \beta \in \mci \} \mbox{ is not a consistent partial type.} 
\end{equation}
Why? $\mx_{f_*} \cap \mb_{\mci} = 0$, 
so the formula $(\exists x) \bigwedge \{ R(x,\bar{a}_{v_\beta}[t])^{\trv(\beta)} : \beta \in \mci \}$ belongs to $\Gamma_0$. 

\br
As we are working in $T_{n,k}$, the inconsistency of (\ref{d:inc-a}) can come from one of two sources (collisions or edges), which we rule out in turn. 


\br
The first possible problem is collision of parameters, i.e. perhaps there are $\beta \neq \gamma \in \mci$ such 
that $\trv(\beta) \neq \trv(\gamma)$ but $\{ a_i[t] : i \in v_\beta \} = \{ a_j[t] : j \in v_\gamma \}$. 
By condition (\ref{f-star-by-ind})(iv) in the inductive construction of $f_*$, we know that 
for each $\beta \in \mci$, $f_*$ extends an element of $\bar{f}_{s_\beta}$. Thus, for each $i \in w_{s_\beta, \zeta_{h(\epsilon)}}$, 
the `minimum collision' functions $\rho_i(f_*)$ from (\ref{defn-of-rho-i}) are well defined. 
Translating,  
\[ \{ a_{\rho_i(f_*)}[t] : i \in v_\beta \} = \{ a_i[t] : i \in v_\beta \} = \{ a_j[t] : j \in v_\gamma \} = \{ a_{\rho_j(f_*)}[t] : j \in v_\gamma \}. \]
Moreover, since the functions $\rho_i$ were constructed to give a (definitive) minimal witness, we have that 
\[ \{ \rho_i(f_*) : i \in v_\beta \} = \{ \rho_j(f_*) : j \in v_\gamma \} \in [W]^k. \] 
Call this set $v$. Let $\delta < \lambda$ be such that $v_\delta = v$ in the enumeration from (\ref{eq-ind2a}).  Recalling the 
definition of the $w_{s,\zeta}$ in (\ref{defn:w}), necessarily $v_\delta \in [w_{s_\beta, \zeta_{h(\epsilon)}}]^k$ and 
$v_\delta \in [w_{s_\gamma, \zeta_{h(\epsilon)}}]^k$, or in other words,   
\[  \delta \in \ndx( w_{s_\beta, \zeta_{h(\epsilon)}} ) \cap \ndx(w_{s_\gamma, \zeta_{h(\epsilon)}}). \]
By definition of $g_{s,\zeta}$ in (\ref{d:float})(c), the collision in each case ensures that 
\begin{equation}
\label{first-x} g_{s_\beta, \zeta_{h(\epsilon)}}(\beta) = g_{s_\beta, \zeta_{h(\epsilon)}}(\delta) \mbox{ and likewise } 
g_{s_\gamma, \zeta_{h(\epsilon)}}(\delta) = g_{s_\gamma, \zeta_{h(\epsilon)}}(\gamma). 
\end{equation}
Recall that we had chosen $s_\beta \in \uu_{\beta, \epsilon}$ and $s_\gamma \in \uu_{\gamma, \epsilon}$ in (\ref{f-star-by-ind})(iii), 
so condition (\ref{d:float})(a) gives that 
\begin{equation}
\label{on-one-hand-x}
 g_{s_\beta, \zeta_{h(\epsilon)}}(\beta) = \trv(\beta) \mbox{ and likewise } \trv(\gamma) = g_{s_\gamma, \zeta_{h(\epsilon)}}(\gamma).
\end{equation}
However, condition (\ref{f-star-by-ind})(iv) in the inductive construction of $f$ ensures that for each $\beta \in \mci$, 
$f \supseteq f^*_{s_\beta, \zeta_{h(\epsilon)}}$. Since $\epsilon$ is fixed, by condition (\ref{defining-f-star})(7), 
\begin{equation}
\label{second-x}
g_{s_\beta, \zeta_{h(\epsilon)}}(\delta) = g_{s_\gamma, \zeta_{h(\epsilon)}}(\delta). 
\end{equation}
By (\ref{first-x}), (\ref{second-x}), and transitivity of equality, 
\begin{equation}
\label{on-the-other-hand-x}
g_{s_\beta, \zeta_{h(\epsilon)}}(\beta) =  g_{s_\beta, \zeta_{h(\epsilon)}}(\gamma). 
\end{equation}
In the presence of our hypothesis that $\trv(\beta) \neq \trv(\gamma)$, equations 
(\ref{on-one-hand-x}) and (\ref{on-the-other-hand-x}) give an obvious contradiction. This contradiction shows that 
collision of parameters cannot be responsible for the inconsistency of the partial type. 

\br
The second possible problem is a background instance (or instances) of $R$ on the parameters, i.e. perhaps there is $w \in [W]^{n}$ 
such that for all $u \in [w]^{k+1}$, $M \models R(\bar{a}_u[t])$, and for each $v \in [w]^k$, there is 
$\beta = \beta(v) \in \mci$ such that $\trv(\beta) = 1$ and $\{ a_i[t] : i \in v_{\beta} \} = \{ a_j[t] : j \in v \}$.  
 
Recall our property $\pr$ for the function $G$ with range $\mu$ (fixed just before defining the equivalence relation $E$ earlier in the proof)
guarantees that: ``for any  $w \in [\lambda]^n$ 
and any $\langle u_v : v \in [w]^k \rangle$ such that $v \in [w]^k$ implies $v \subseteq u_v \in [\lambda]^{<\aleph_0}$,  
\emph{if} $G \rstr \langle u_v : v \in [w]^k \rangle$ is constant, 
\emph{then} for some $v \in [w]^k$ we have that $w \subseteq u_v$.''  
Apply this in the case where $u_v = w_{s_{\beta(v)}, \zeta_{h(\epsilon)}}$.  Because $\epsilon$ is fixed, the value 
$\xi = \xi_{h(\epsilon)}$ of $G$ on these sets is constant. 
Thus, there is some $\beta_* \in \mci$ such that $w \subseteq w_{\beta_*, \zeta}$.  
In other words, the relevant near-complete hypergraph is already contained in the base set of one of our consistent partial types. 

Now the argument is similar to that of the ``collision'' problem treated above. 
Fix for awhile $v \in [w]^k$ and $\beta = \beta(v)$. Let $\delta < \lambda$ be such that $v = v_\delta$. 
Again, for each $i \in w_{s_\beta, \zeta_{h(\epsilon)}}$ the functions $\rho_i(f_*)$ are well defined and entail that 
\[ \{ \rho_i(f_*) : i \in v_{\beta} \} = \{ \rho_j(f_*) : j \in v \}. \]
Thus, by (\ref{d:float})(c), 
\[ g_{s_{\beta}, \zeta_{h(\epsilon)}}(\beta) = g_{s_{\beta}, \zeta_{h(\epsilon)}}(\delta). \] 
Since $w \subseteq w_{s_{\beta_*}, \zeta_{h(\epsilon)}}$ and $v \in [w]^k$, we have also that $\delta \in \dom(g_{s_{\beta_*}, \zeta_{h(\epsilon)}})$. 
Again by (\ref{f-star-by-ind})(iv) and (\ref{defining-f-star})(7), we have that 
\[ g_{s_{\beta_*}, \zeta_{h(\epsilon)}}(\delta) = g_{s_{\beta}, \zeta_{h(\epsilon)}}(\delta) = \trv(\beta) = 1. \] 
As $v \in [w]^k$ was arbitrary, this shows that $r_{\beta_*}(x)$ includes $\{ R(x,\bar{a}_v) : v \in [w]^k \}$.  
In light of our assumption that $M \models R(\bar{a}_u[t])$ for all $u \in [w]^{k+1}$, this contradicts $r_{\beta_*}$ being a consistent partial type.  
This shows that an occurrence of $R$ on the parameters cannot be responsible for inconsistency of the partial type. 

We have ruled out the only two possible causes of inconsistency for (\ref{d:inc-a}). 
This contradiction proves that the situation of (\ref{for-a-contradiction-1}) never arises. This completes the proof that $\bar{\mb}^\prime$ 
is a multiplicative refinement of $\bar{\mb}$. 

This completes the proof of Theorem \ref{key-sat-1}.  
\end{proof}

\begin{concl} \label{c-1}
Suppose that for some ordinal $\alpha$ and integers $\ell$, $k$,
\begin{enumerate}
\item $\ell < k$
\item $T = T_{k+1,k}$ 
\item $\mu = \aleph_\alpha$, $\lambda = \aleph_{\alpha + \ell}$
\end{enumerate}
Then there is a regular $(\lambda, \mu)$-perfect ultrafilter on $\lambda$ which is good for $T$ but not  
$\mu^{++}$-good for any non-low or non-simple theory. 
\end{concl}

\begin{proof} 
Theorem \ref{t:exists} gives a $ (\lambda, \mu)$-perfected ultrafilter which is not $\mu^{++}$-good for 
any non-simple or non-low theory. For the saturation condition, 
Lemma \ref{z4} proves that $\pr_{n,k}(\lambda, \mu)$ holds for these cardinals 
so the hypotheses of Theorem \ref{key-sat-1} are satisfied. 
\end{proof}

\br
\br

\section{The non-saturation condition}

In this section we prove the complementary result to Theorem \ref{key-sat-1}, by connecting 
non-saturation of $T_{k+1, k}$ to existence of large free sets in set mappings.

\begin{claim} \label{f2a}
Suppose that:  
\begin{enumerate}
\item for some ordinal $\alpha$ and integers $2 \leq k < \ell$, $\mu = \aleph_\alpha$, $\lambda = \aleph_{\alpha + \ell}$, 
\\ or just: $(\lambda, k, \mu^+) \rightarrow k+1$
\item $\ba = \ba^1_{2^\lambda, \mu}$ 
\item $\de_*$ is an ultrafilter on $\ba$
\item $T = T_{k+1,k}$ 
\end{enumerate}
Then $\de_*$ is not $(\lambda, T)$-moral.
\end{claim}

\begin{rmk}
Note that there is no mention of optimality or perfection of the ultrafilter. 
The only factor is the distance of $\lambda$ and $\mu$ as reflected in the Boolean algebra $\ba$ 
$($or what amounts to the size of a maximal antichain at the ``transfer point'' in Theorem \ref{t:separation}$)$. 
\end{rmk}

\begin{proof} 

Our strategy will be to build a sequence $\bar{\mb}$ of elements of $\ba^+$ and prove that it is a possibility pattern for $T$ but does not have a multiplicative refinement.  
We continue with much of the notation and terminology of the previous section. 

By Theorem \ref{t:yes} above (and monotonicity), for $k \leq m = \ell - 1$, $(\aleph_{\alpha + m + 1}, k, \aleph_{\alpha+1}) \rightarrow k+1$, 
so we can apply Claim \ref{f2ax} to $(\lambda, k, \mu^+)$. [Notice that $\mu^+$ here replaces $\mu$ there.]
Thus, we may fix a model $M$ of $T_{k+1,k}$ 
with $\lambda$ distinguished elements ${\bar{b} = \langle b_\alpha : \alpha < \lambda \rangle}$ with the following property. 
Let 
\[ \bad = \{ w \in [\lambda]^{k+1} : M \models R(\overline{b}_v)  \} \] 
noting that by choice of $T$, $\bad \subsetneq [\lambda]^{k+1}$. The property is that 
whenever $F: [\lambda]^{k} \rightarrow [\lambda]^{\leq\mu}$ is a strong set mapping,  
for some $w \in \bad$ we have 
\[ (\forall v \in [w]^{k})(w \not\subseteq F(v)). \]

Without loss of generality we may extend $M$ to be $\lambda^+$-saturated. 
For the remainder of the proof, fix a choice of ordinals 
$\langle \alpha_w : w \in \bad \rangle$ with no repetitions, where each $\alpha_w < 2^\lambda$. 
Choose also for each $w \in \bad$ a corresponding function $g_w \in \fin_{\mu}(\alpha_*)$ such that $\dom(g_w) = \{ \alpha_w \}$ and $\mx_{g_w} = \emptyset \mod \de_*$. 

\br

\noindent\emph{Overview in a special case.}
Before giving the construction in the generality of the Boolean algebra $\ba$, we describe for the reader the picture in the special case where 
we consider an ultrapower $N = M^I/\de$ where $\de$ is built from a regular filter $\de_0$ and $\ba$ is identified with some independent family $\eff \subseteq {^I\mu}$
of cardinality $2^\lambda$.  
What we would like to do is choose a set $A$ of size $\lambda$ in the ultrapower which is an empty graph in $N$, i.e. for all $u \in [A]^{k+1}$, 
$N \models \neg R(\bar{a}_u)$. As a result, the type $p(x) = \{ R(x,\bar{a}_{v}) : v \in [A]^{k+1} \}$ will be a consistent partial type in $N$.  
However, by judicious choice of the parameter set $A$, we will be able to show that $p$ cannot be realized. To do this we need to ensure that edges appear 
on the projections of $A$ to the index models, but not too many and not too often. 

We begin with the idea that for each $i < \lambda$, $a_i$ is the equivalence class in $N$ of the sequence which is constantly equal to $b_i$. We then 
essentially doctor this sequence by winnowing $\mcp$, i.e. erasing some of the edges. 
Formally, of course, at each index $t$ we choose a sequence $\langle b^\prime_i[t] : i < \lambda \rangle$ of distinct elements of $M$ 
(using the fact that $M$ is universal for models of $T$ of size $\leq \lambda$) such that for all $w \subseteq \lambda$, 
if $M \models R(\bar{b}^\prime_w)$ then $M \models R(\bar{b}_w)$, but not necessarily the inverse. We will 
then set $a_i = \langle b^\prime_i[t] : t \in I \rangle/\de$ for each $i < \lambda$. How to winnow edges?   
Following the notation of the proof of \ref{key-sat-1}, fix an enumeration of $[\lambda]^k$ as $\langle v_\beta : \beta < \lambda \rangle$ without repetition, 
so the eventual type will be enumerated by $\{ R(x,\bar{a}_{v_\beta}) : \beta < \lambda \}$. Let $\Omega = [\lambda]^{<\aleph_0}$. 
For each $s \in \Omega$, let the `critical set' $\crs(s)$ be the set of $w \in \mcp$ such that each $v \in [w]^k$ is $v_\beta$ for some $\beta \in w$.  
(Note that this is generally weaker than saying that $w \subseteq \ver(s)$.)
The rule is that for each $t \in I$, and each $w \in \mcp$, we leave an edge on $\{ b^\prime_i : i \in w \}$ 
if and only if $t \in \mx_{g_w}$.  By the choice of $g_w$, no edge will persist in the ultrapower, so $\langle a_i : i < \lambda \rangle$ is an empty graph in $N$.  
It remains to prove the type is not realized.  
Before giving this argument, we carry out the construction just described in the generality of 
the Boolean algebra. (The type just described easily converts to a possibility pattern using the \los map as in (\ref{by-the-los-map}) p. \pageref{by-the-los-map},   
so we may conclude this argument using the more general proof.) 

\br

\noindent\emph{General proof.}
Let $\langle v_\alpha : \alpha < \lambda \rangle$ list $[\lambda]^k$ without repetition. 
For $s \subseteq \lambda$, let $\ver(s) = \bigcup \{ v_\beta : \beta \in s \} \in [\lambda]^{<\aleph_0}$ collect the indices 
for all relevant vertices. 
Let $\Omega = [\lambda]^{<\aleph_0}$. 
For each $s \in \Omega$, let 
\[ \mb_s = 1_\ba - \bigcup \{ \mx_{g_w} : w \in \bad \mbox{ and } [w]^k \subseteq \{ v_\beta : \beta \in s \} \}. \]
Essentially, we omit the formal representative of any bad configuration once our type fragment $s$ includes indices for all of 
the edges (in the type) connecting to it.  

Let us show that $\langle \mb_s : s \in \Omega \rangle$ is a possibility pattern for $T_{k+1,k}$. 
Fix for awhile $s \in \Omega$ and $\mc \in \ba^+$. 
Decreasing $\mc$ if necessary, we may assume that for any $w \in \bad \cap [\ver(s)]^{k+1}$ 
either $\mc \leq \mx_{g_w}$ or $\mc \leq 1-\mx_{g_w}$. It follows that for any $s^\prime \subseteq s$ either 
$\mc \leq \mb_{s^\prime}$ or $\mb \leq 1-\mb_{s^\prime}$. 

To satisfy Definition \ref{d:poss}, we now need to choose parameters $b^\prime_i \in M$ for $i \in \ver(s)$ such that: 
$\bar{b}^\prime = \langle b^\prime_i : i \in \ver(s) \rangle$ is without repetition and for any $s^\prime \subseteq s$, 
\[  M \models (\exists x)\bigwedge_{\beta \in s^\prime} \vp(x;\bar{b}^\prime_{v_\beta}) ~~ \mbox{ ~~iff~~ } ~~ \mc \leq \mb_{s^\prime}.  \]
We can do this by choosing our parameters so that for any $i_0, \dots, i_{k-1} \in \ver(s)$ we have 
$\langle b^\prime_{i_\ell} : \ell < k \rangle \in R^M$ if and only if: $| \{ i_\ell : \ell < k \} | = k$ [i.e. they are distinct] and  $\{ i_\ell : \ell < k \} \in \mcp$  
and $\mc \leq \mx_{g_{\{i_\ell : \ell < k \}}}$. Note that there \emph{is} such a sequence of parameters in the monster model 
(forgetting edges on the $\bar{b}$ as described above) so it suffices to show such a sequence works. 
If $\mc \leq \mb_{s^\prime}$, then by definition of $\mb_{s^\prime}$, there is no $w \in \bad$ such that 
$[w]^k \subseteq \{ v_\beta : \beta \in s^\prime \}$ and $\mc \leq \mx_{g_w}$. So there are never 
enough edges on the parameters to produce an inconsistency in the set  
\[ \{ R(x;\bar{b}^\prime_{v_\beta}) : \beta \in s^\prime \}. \]
If $\mc \cap \mb_{s^\prime} = 0_\ba$, then because $\mc \in \ba \setminus \{ 0_\ba \}$, it must be that $\mb_{s^\prime} \neq 1_\ba$.  
By definition of the sequence $\bar{\mb}$, there is $w \in \bad$ with $[w]^k \subseteq \{ v_\beta : \beta \in s^\prime \}$ and (since 
$\mc$ decides all relevant edges) $\mc \leq \mx_{g_w}$. Then $M \models R(\bar{b}^\prime_w)$.  Recalling that 
\[ \{ R(x;\bar{b}^\prime_{v_\beta}) : \beta \in s^\prime \} \supseteq \{ R(x;\bar{b}^\prime_{v}) : v \in [w]^k \} \]
the left hand side cannot be consistent. This completes the proof that $\bar{\mb}$ is a possibility pattern.

\br

\noindent \emph{No multiplicative refinement.}
Now let us assume for a contradiction that $\langle \mb^\prime_s : s \in \Omega \rangle$ is a multiplicative refinement of the possibility pattern just described. 
That is, $s_1, s_2 \in \Omega$ implies $\mb^\prime_{s_1} \cap \mb^\prime_{s_2} = \mb^\prime_{s_1 \cap s_2}$ and for each $s \in \Omega$, $\mb^\prime_{s} \leq \ma_{s}$. 
As each $\mb^\prime_{\{\beta\}} \in \ba^+$, we may write 
$\mb^\prime_{\{\beta\}} = \bigcup \{ \mx_{h_{\beta, i}} : i < i(\beta) \leq \mu \}$ where $\langle h_{\beta,i} : i < i(\beta) \rangle$ is a set of pairwise 
inconsistent functions from $\fin_{\mu}(2^\lambda)$. Let $S_\beta = \bigcup \{ \dom(h_{\beta,i}) : i < i(\beta) \}$, so 
$S_\beta \subseteq 2^\lambda$ has cardinality $\leq \mu$. 

First, we show that for each $w \in \bad$ the domain of $g_w$ is detected by the supports of at least one of the the $k$-element subsets of $w$. 

\begin{subclaim} \label{subclaim-m}
If $w \in \bad \subseteq [\lambda]^{k+1}$ then $\alpha_{w} \in \bigcup \{ S_\beta : v_\beta \in [w]^k \}$. 
\end{subclaim}

\begin{proof} 
Let $x = \{ \beta : v_\beta \in [w]^k \} \in [\lambda]^{\binom{m}{k}}$. Since $\overline{\mb}^\prime$ is multiplicative, 
\[ \mb^\prime_x = \bigcap \{ \mb^\prime_\beta : \beta \in x \} \]
Let $f \in \fin_{\mu}(2^\lambda)$ be such that $\mx_{f} \leq \mb^\prime_x$. Then $\mx_f \leq \mb^\prime_{\{\beta\}}$ for each $\beta \in x$. 
Letting $g = f \rstr \bigcup \{ S_\beta : v_\beta \in [w]^k \} = \bigcup \{ S_\beta : \beta \in x \}$, we have that 
$v_\beta \in [w]^k \implies \mx_{g} \leq \mb^\prime_{\{\beta\}}$. This implies that $\mx_{g} \leq \mb^\prime_x \leq \mb_x$ 
because $\bar{\mb}^\prime$ refines $\bar{\mb}$. 
By definition, 
\[ \mb_x = 1_\ba - \bigcup \{ \mx_{g_u} : u \in \bad \mbox{ and } [u]^k \subseteq \{ v_\beta : \beta \in x \} \}. \]
So as $[w]^k \subseteq \{ v_\beta : \beta \in x \}$, necessarily $\mx_{g} \cap \mx_{g_{w}} = 0_\ba$. 
Since our Boolean algebra $\ba$ was generated freely, it must be that $\dom(g_{w}) \cap \dom(g) \neq \emptyset$, 
but $\dom(g_{w}) = \{ \alpha_{w} \}$. This shows that $\alpha_{w} \in \bigcup \{ S_\beta: v_\beta \in [w]^k \}$ as desired. 

\noindent\emph{This proves Subclaim \ref{subclaim-m}}. \hfill \end{proof}

We resume our proof by contradiction. 
Define a strong set mapping $ F: [\lambda]^k \rightarrow [\lambda]^{\leq \mu} $ 
by: if $v \in [\lambda]^k$ let $\beta$ be such that $v = v_\beta$, and let 
\[ F(v) = \bigcup \{ w \in [\lambda]^m : w \in \bad \mbox{ and } \alpha_w \in S_\beta \}. \]
Then  $F(v)$ is well defined, $F(v) \subseteq \lambda$, and $|F(v)| \leq \mu$ for $v \in [\lambda]^k$. 
(Recall that $\langle \alpha_w : w \in \bad \rangle$ is without repetition.) 
Now for all $w \in \bad \subseteq [\lambda]^{k+1}$, there is $v = v_\beta \in [w]^k$ such that $\alpha_w \in S_\beta$. 
{Thus} $w \subseteq F(v)$. We have proved that for all $w \in \bad$,  
\[ (\exists v \in [w]^k) ( w \subseteq F(v) ). \]
This is a contradiction, so the possibility pattern $\overline{\mb}$ does not have a solution. 
Thus, $\de_*$ cannot be moral for $T_{k+1,k}$. This completes the proof of Claim \ref{f2a}. 
\end{proof}

\begin{concl} \label{c-2}
Suppose that for some ordinal $\alpha$ and integers $\ell$, $k$,
\begin{enumerate}
\item $2 \leq k < \ell$
\item $T = T_{k+1,k}$ 
\item $\mu = \aleph_\alpha$, $\lambda = \aleph_{\alpha + \ell}$
\item $\ba = \ba^1_{2^\lambda, \mu}$ 
\item $\de_*$ is \emph{any} ultrafilter on $\ba$
\item $\de_1$ is \emph{any} regular ultrafilter on $\lambda$ built from $(\de_0, \ba, \de_*)$
\end{enumerate}
Then $\de_1$ is not good for $T$. In particular, if $\de_1$ is a $(\lambda, \mu)$-perfected ultrafilter on $\lambda$,  
then 
$\de_1$ is not good for $T$. 
\end{concl}

\begin{proof} 
By Claim \ref{f2a} and Theorem \ref{t:separation}.  Note that if we allow $\ell = k = 1$, $T_{k+1, k}$ is not simple so we can likewise 
avoid saturation of $T$. 
\end{proof} 

\br

\section{Infinitely many classes} \label{s:infinite}

We emphasize that all results in this section are in ZFC. 

\begin{theorem} \label{t:p2a}
Suppose $\mu = \aleph_\alpha$ and $\lambda = \aleph_{\alpha + \ell}$ for $\alpha$ an ordinal and $\ell$ a nonzero integer. 
Let $\de$ be a $(\lambda, \mu)$-perfected ultrafilter on $\lambda$. Then for any $2 \leq k < \omega$: 
\begin{enumerate}
\item[(a)] If $k < \ell$, then $\de$-ultrapowers of models of $T_{k+1, k}$ are not $\lambda^+$-saturated. 

\item[(b)] If $\ell < k$, then $\de$-ultrapowers of models of $T_{k+1, k}$ are $\lambda^+$-saturated. 
\end{enumerate} 
\end{theorem}

\begin{proof}
(1) Conclusion \ref{c-2}. 

(2) Conclusion \ref{c-1}. 
\end{proof}

In fact, by the proofs, more is true: 

\begin{concl} 
Suppose we are given: 
\begin{enumerate}[label=\emph{(\alph*)}]
\item for some ordinal $\alpha$ and integer $\ell$, 
$\mu = \aleph_\alpha$, $\lambda = \aleph_{\alpha + \ell}$
\item $\de_1$ is built from $(\de_0, \ba_{2^\lambda, \mu}, \de)$ 
\item $T = T_{k+1,k}$ 
\end{enumerate}
Then:

\begin{enumerate}
\item if $k < \ell$, $\de_1$ is not $(\lambda^+, T)$-good. 

\item if $\ell < k$ and in addition $\de$ is $(\lambda, \mu)$-perfect,  
$\de_1$ is $(\lambda^+, T)$-good. 
\end{enumerate}
\end{concl}

\begin{theorem} 
\label{t:list} 
For any $k_* > 2$ and ordinal $\alpha$ there is a regular ultrafilter $\de$ on $\aleph_{\alpha + k_*}$ such that 
\begin{enumerate}
\item if $k_* < k_2$ then $\de$ is good for $T_{k_2, k_2+1}$
\item if $k_1 < k_*$ then $\de$ is not good for $T_{k_1, k_1+1}$.
\end{enumerate}
\end{theorem}

\begin{proof}
By Theorem \ref{t:p2a} and \S \ref{s:various} Theorem \ref{t:exists}. 
\end{proof}

We now recall the definition of Keisler's order. 
For a current account of what is known, see \cite{MiSh:998} and for further intuition, see the introductory 
sections of \cite{MiSh:996}. Note that this allows us to compare any two theories, regardless of language. 

\begin{defn} \emph{(Keisler's order, Keisler 1967 \cite{keisler})} \label{k-order}
Let $T_1, T_2$ be complete countable theories. 
We write $T_1 \tlf T_2$ if: for any $\lambda \geq \aleph_0$, any $M_1 \models T_1$, any $M_2 \models T_2$ and any regular ultrafilter on $\lambda$,
\[ \mbox{if ${M_2}^\lambda/\de$ is $\lambda^+$-saturated then ${M_1}^\lambda/\de$ is $\lambda^+$-saturated.} \]
Here ``regular'' entails that the relation $\tlf$ is independent of the choice of $M_1$, $M_2$. 
\end{defn}

\begin{cor} \label{c42} 
Let $\tlf$ mean in Keisler's order. Then:
\begin{enumerate}
\item If $2 \leq k_1 < k_2$ then  
\[ T_{k_1, k_1+1} \not\tlf T_{k_2, k_2+1}. \] 
\item Keisler's $($partial$)$ order contains either an infinite descending chain or an infinite antichain within the simple unstable theories. 
\end{enumerate}
\end{cor}

\begin{proof} (1) is immediate by \ref{t:list} and (2) follows by Ramsey's theorem. 
\end{proof} 

Note that Keisler's order is a partial order on equivalence classes of theories, 
and the following theorem proves \emph{existence} of an infinite descending chain 
in this partial order already within the simple unstable rank one theories; there may indeed be additional structure.  

\begin{theorem} 
\label{t:seq} 
There is an infinite descending sequence of simple rank 1 theories in Keisler's order. More precisely, 
there are simple theories $\{ T^*_n : n < \omega \}$ with trivial forking such that, writing
\begin{itemize}
\item $\mct_A$ for the class of theories without fcp
\item $\mct_B$ for the class of stable theories with fcp
\item $\mct_C$ for the minimum unstable class, i.e. the Keisler-equivalence class of the random graph
\item $\mct_{max}$ for the Keisler-maximal class, i.e. the Keisler-equivalence class of linear order $($or $SOP_2$$)$
\item and $\mct_n$ for the Keisler-equivalence class of $T^*_n$
\end{itemize} 
for all $ m < n < \omega$ we have:
\[ \mct_{A} \tlfn \mct_{B} \tlfn \mct_{C} \tlfn \hspace{15mm} \cdots \cdots \cdots \mct_{n} \tlfn \mct_m \tlfn \cdots \tlfn \mct_2 \tlfn \mct_1 \tlfn \mct_0 \tlfn \hspace{8mm} \mct_{max}. \]
\end{theorem}

\begin{proof}
The structure of the order on $\mct_{A}$, $\mct_{B}$, $\mct_{C}$, $\mct_{max}$ was known, see \cite{MiSh:996} \S 4. 
To obtain the infinite descending chain, let $T^*_n$ be the disjoint union of the theories $T_{k,k+1}$ for $k > n$. 
Here ``disjoint union'' is understood naturally, for instance, the theory of the model $M$ formed by taking the disjoint union of models $M_k \models T_{k, k+1}$
in disjoint signatures.  Clearly, $k^\prime > k$ implies  $T^*_{k^\prime} \tlf T^*_{k}$ and $\tlfn$ is by Theorem \ref{t:list}. 
This completes the proof. 
\end{proof}

\end{document}